\documentclass[review,hidelinks,onefignum,onetabnum]{siamart250106}

\usepackage{lipsum}
\usepackage{amsfonts}
\usepackage{graphicx}
\usepackage{epstopdf}
\ifpdf
  \DeclareGraphicsExtensions{.eps,.pdf,.png,.jpg}
\else
  \DeclareGraphicsExtensions{.eps}
\fi

\newsiamremark{remark}{Remark}
\newsiamremark{hypothesis}{Hypothesis}
\crefname{hypothesis}{Hypothesis}{Hypotheses}
\newsiamthm{claim}{Claim}

\Crefname{ALC@unique}{Line}{Lines}

\usepackage{amsopn}

\usepackage{amsfonts}
\usepackage[T1]{fontenc}
\usepackage{tensors}
\usepackage{algpseudocode}
\usepackage{bookmark}
\usepackage{diagbox}

\usepackage{tikz}
\usetikzlibrary{positioning}
\usetikzlibrary{cd}
\usetikzlibrary{3d, calc, decorations.pathreplacing, positioning}

\usepackage{booktabs}
\usepackage{multirow}
\usepackage{tabularx}
\usepackage{array}
\usepackage{stmaryrd}

\graphicspath{{fig/}}

\usepackage{latexsym,array,delarray,amssymb,epsfig}
\usepackage{amsmath}
\usepackage{listings}
\usepackage{float}
\newcommand{\R}{\mathbb{R}}

\usepackage{lipsum}
\usepackage{amsfonts}
\usepackage{graphicx}
\usepackage{algorithm,algpseudocode}
\usepackage{textcomp}
\usepackage{xcolor}
\usepackage[nolist,nohyperlinks]{acronym}
\usepackage{tikz}
\usetikzlibrary{matrix, positioning, 3d, patterns}
\usepackage{multirow}
\usepackage{array}
\usepackage{booktabs}
\usepackage[section]{placeins}
\usepackage{siunitx}
\usepackage{subcaption}
\usepackage{balance}
\usepackage{hyperref}
\usepackage{xspace}

\usepackage{bm}
\newcommand{\V}[2][]{\bm{#1{\mathbf{\MakeLowercase{#2}}}}} 		
	
\newcommand{\kr}{\odot}
\newcommand{\kron}{\otimes}
\newcommand{\bmat}[1]{\begin{bmatrix} #1 \end{bmatrix}}

\newcommand{\col}[1]{\texttt{col}(#1)}
\newcommand{\size}[1]{\texttt{size}(#1)}

\algnewcommand{\IfThenElse}[3]{\State \algorithmicif\ #1\ \algorithmicthen\ #2\ \algorithmicelse\ #3} 
\algnewcommand{\IfThen}[2]{\State \algorithmicif\ #1\ \algorithmicthen\ #2} 

\newcommand{\h}[1]{\mathcal{H} \left(#1\right)}
\renewcommand{\v}[1]{\mathcal{V}\left(#1\right)}
\makeatletter
\newsavebox{\@brx}
\newcommand{\llangle}[1][]{\savebox{\@brx}{\(\m@th{#1\langle}\)}%
  \mathopen{\copy\@brx\kern-0.5\wd\@brx\usebox{\@brx}}}
\newcommand{\rrangle}[1][]{\savebox{\@brx}{\(\m@th{#1\rangle}\)}%
  \mathclose{\copy\@brx\kern-0.5\wd\@brx\usebox{\@brx}}}
\makeatother

\allowdisplaybreaks
\nolinenumbers

\title{Adaptive Randomized Tensor Train Rounding using Khatri-Rao Products\thanks{This work was funded in part by the Department of Energy, through awards DE-SC0025262 and DE-SC0025394; the European Research Council (ERC) under the European Union’s Horizon 2020 research and innovation program (grant agreement No 810367); the National Science Foundation under grant CCF-1942892; and the Ada Lovelace Centre Programme at the Scientific Computing Department, STFC.}}
\author{Hussam Al Daas\thanks{Scientific Computing Department, STFC, Rutherford Appleton Laboratory, Didcot OX11 0QX, UK (hussam.al-daas@stfc.ac.uk).} \and Grey Ballard\thanks{Department of Computer Science, Wake Forest University, Winston-Salem, NC 27106 USA (ballard@wfu.edu and vermabd@wfu.edu).} \and  Laura Grigori\thanks{Institute of Mathematics, EPFL, Lausanne, and PSI Center for Scientific Computing, Theory and Data, Villigen PSI, Switzerland (laura.grigori@epfl.ch)} \and Mariana Martinez Aguilar\thanks{Institute of Mathematics, EPFL, Lausanne, Switzerland (mariana.martinezaguilar@epfl.ch)} \and Arvind K.\ Saibaba\thanks{Department of Mathematics, North Carolina State University, Raleigh, NC 27695 USA (asaibab@ncsu.edu)} \and Bhisham Dev Verma\footnotemark[3]}

\makeatletter
\newcommand*{\addFileDependency}[1]{
  \typeout{(#1)}
  \@addtofilelist{#1}
  \IfFileExists{#1}{}{\typeout{No file #1.}}
}
\makeatother

\ifpdf
\hypersetup{
  pdftitle={TT rounding using KRP},
  pdfauthor={Authors}
}
\fi

\begin{document}

\maketitle

\begin{abstract}
Approximating a tensor in the tensor train (TT) format has many important applications in scientific computing. 
Rounding a TT tensor involves further compressing a tensor that is already in the TT format. This paper proposes new randomized algorithms for TT-rounding that uses sketches based on Khatri-Rao products (KRP). When the TT-ranks are known in advance, the proposed methods are comparable in cost to the sketches that used a sketching matrix in the TT-format~\cite{al2023randomized}. However, the use of KRP sketches enables adaptive algorithms to round the tensor in the TT-format within a fixed user-specified tolerance. An important component of the adaptivity is the estimation of error using KRP sketching, for which we develop theoretical guarantees. We report numerical experiments on synthetic tensors, parametric low-rank kernel approximations, and the solution of parametric partial differential equations. The numerical experiments show that we obtain speed-ups of up to $50\times$ compared to deterministic TT-rounding. Both the computational cost analysis and numerical experiments verify that the adaptive algorithms are competitive with the fixed rank algorithms, suggesting the adaptivity introduces only a low overhead. 
\end{abstract}



\section{Introduction} \label{sec:introduction}
A growing range of problems in scientific computing and data analysis involve high-dimensional data represented as tensors---multidimensional generalizations of vectors and matrices. When considering tensors, one must consider the specific tensor format to be used and the associated notions of ranks associated with tensors. Some examples of tensor formats are CP decomposition, Tucker, Tensor Train (TT), Hierarchical Tucker, etc., which have been reviewed in~\cite{ballard2025tensor,kolda2009tensor}. Inspired by the matrix-product-state (MPS), which is widely used in the computational quantum chemistry community \cite{Whi92}, Oseledets ~\cite{oseledets_tensor-train_2011} introduced the TT format to the numerical analysis community. The cost of storing the tensor in the TT-format depends linearly on the number of dimensions. A major challenge is to find a suitable approximation to a tensor in the TT-format with sufficiently small TT-ranks. There are several algorithms to convert a tensor into TT-format, such as TT-SVD~\cite{oseledets_tensor-train_2011} and cross-approximation algorithms \cite{DolS20,SavO11,Sav14}.

This paper is focused on TT-rounding, which is a special case of TT approximation. Rounding a TT-tensor involves finding a compressed representation of a tensor which is already in TT-format, but perhaps with large ranks. This is an important problem with applications to parametric low-rank kernel approximation~\cite{khan_parametric_2025}, quantum physics~\cite{camano_successive_2025}, high dimensional partial differential equations~\cite{APTT}, thermal radiation transport~\cite{gorodetsky2025thermal}, and the solution of linear tensor equations~\cite{bucci2025randomized}.

\paragraph{Contributions} In this paper, we derive new randomized algorithms for rounding a tensor in the TT-format. The specific contributions of our work are:
\begin{enumerate}
    \item We propose two new randomized algorithms (Sections~\ref{sec:fixed_rank} and~\ref{sec:adaptive_rank}) for rounding a TT-tensor. Both algorithms use a random tensor for sketching based on the Khatri-Rao product (KRP). The first algorithm assumes the target ranks are known \textit{a priori}, whereas the second algorithm takes a prespecified tolerance for the relative error and adaptively determines the ranks.
    \item We analyze, in detail, the computational complexity of both algorithms (Sections~\ref{sec:fixed_rank} and~\ref{sec:adaptive_rank}). While the proposed methods are comparable to existing randomized rounding algorithms, the use of KRPs enables adaptivity, which is missing in existing work. 
    \item The adaptive randomized algorithm uses a termination criterion based on a randomized Frobenius norm estimator. We also provide an analysis for this estimator (Section~\ref{ssec:analysis}). 
    \item Numerical experiments (Section~\ref{sec:experiments}) on three different sets of applications---synthetic tensors, parametric kernel low-rank approximation, and numerical solution of a parametric partial differential equation---demonstrate the performance of the method. In these applications, we observe speedups of up to $50\times$ compared to the deterministic TT-rounding techniques.
\end{enumerate}
Additionally, the code required to reproduce our results is available at \url{https://github.com/bhisham123/TT-Rounding-using-KRP-structure.git}.

\paragraph{Related Work} 

While there have been several developments in structured sketching for various tensor formats \cite{bucci_randomized_2024, haselby2025fast, hashemi_rtsms_2025, NEURIPS2022_fe91414c, yu_practical_2024}, we focus our review on TT-rounding. 
Oseledets proposed a deterministic method for rounding in TT format \cite{oseledets_tensor-train_2011}.  
Randomized algorithms for computing the TT decomposition can be found in \cite{ahmadi-asl_randomized_2020,che_randomized_2019,che2024randomized,huber_randomized_2017} with an adaptive implementation in \cite{che_randomized_2019}, but these references focus on tensor compression and not rounding, which is the focus of the present paper. 
A method for TT-rounding using the Gram matrices was proposed in~\cite{al2022parallel}. 

Randomized algorithms for TT-rounding were proposed in~\cite{al2023randomized}. Of these, the most relevant to this discussion are the Orthogonalize-then-Randomize (Orth-Rand) and Randomize-then-Orthogonalize (Rand-Orth).  A more detailed description of these algorithms is given in Sections~\ref{supp-sec:otr} and~\ref{supp-sec:rto}.  The Rand-Orth approach uses sketching matrices based on the TT-format and is applicable for only the fixed rank case. This approach is much more computationally efficient compared to the Orth-Rand approach, in which the computational cost is dominated by the initial orthogonalization phase. The Orth-Rand approach has an adaptive variant which enables a comparison to our proposed algorithms. Our approach uses sketches using KRPs and enables adaptivity in the TT-rounding. By also using TT-like structured matrices, Kressner et al. \cite{kressner_streaming} proposed a two-sided, Nystr\"om-like approximation for streaming TT approximation. However, this approach also uses fixed ranks. A comparison of many TT-compression techniques in the context of thermal radiation transport is given in~\cite{gorodetsky2025thermal}.

A recent paper \cite{camano_successive_2025} also uses randomization with KRP structure to compute single matrix-product-operator times matrix-product-state (MPO-MPS) without considering sums of products or particular structures. 
Our work is similar to this paper in that our approach also uses sketching matrices with KRP structure and also uses adaptivity. 
However, our approach can be applied more generally to any TT input and does not specifically exploit the MPO-MPS product structure. 
We also specialize our approach to exploit the structure of inputs that are sums of TT tensors, as discussed in Section~\ref{supp-sec:sum}.
A detailed comparison between the two approaches is outside the scope of the current paper, and we leave it to future work to develop a version of our algorithm that exploits the specific MPO-MPS structure.

The use of KRP sketches for TT-format was also explored in~\cite{bucci2025randomized}, but it was not applied to the rounding step. Recent work by~\cite{camano2025faster,che2024randomized,saibaba2025improved} has advanced the theoretical understanding of sketching using KRP matrices.

\section{Background} \label{sec:background}

\subsection{Notation} \label{subsec:notation}
We use boldface Euler script letters to denote tensors (e.g., $\Tn{X}$) and boldface capital letters for matrices (e.g., $\Mx{A}$). For indexing the entries tensors, we adopt MATLAB-like conventions. For example, we denote the entries of a three-way tensor $\Tn{X}$ as $\Tn{X}(i, j, k)$. Fixing all but one index results in a fiber (a vector), where $\Tn{X}(:,j,k)$ represents a column fiber, $\Tn{X}(i,:,k)$ represents a row fiber, and $\Tn{X}(i,j,:)$ represents a tube fiber. Fixing all but two indices results in slices. For a three-way tensor $\Tn{X}$, $\Tn{X}(i,:,:)$ is referred to as the $i$-th horizontal slice, $\Tn{X}(:,j,:)$ is the $j$-th lateral slice, and $\Tn{X}(:,:,k)$ is the $k$-th frontal slice. For tensors of order higher than three, slices and fibers do not have specific names. We alternatively use the terms \textit{way} or \textit{order} to denote the modes or dimensions of a tensor. $\Mx{X}_{(k)} \in \R^{n_k \times (N \setminus n_k)}$ denotes the mode-$k$ unfolding (or matricization) of a tensor $\Tn{X} \in \R^{n_1 \times \cdots \times n_d}$ where $N = n_1 \cdots n_d$. The columns of the mode-$k$ unfolding consist of the respective mode fibers. $\Mx{X}_{(1:k)} \in \R^{n_1 \cdots n_k \times n_{k+1} \cdots n_d}$ refers to a matrix unfolding where first $k$ modes contributes to row index and remaining $(d-k)$ modes to column index. We call it the unfolding of the first $k$-modes of the tensor. Mode-$k$ unfolding is its special case. For a matrix $\Mx{A} \in \R^{m \times n_k}$, the mode-$k$ product of a tensor $ \Tn{X}$ with $\Mx{A}$, denoted $ \Tn{Y} = \Tn{X} \times_k \Mx{A}$, represents a contraction along mode $k$. This operation can be defined via the mode-$k$ unfolding as  follows: $ \mathbf{Y}_{(k)} = \Mx{A} \Mx{X}_{(k)}$. The norm of a tensor $\Tn{X} \in \R^{n_1 \times \cdots \times n_d}$ is defined as $\|\Tn{X}\| = \left( \sum_{i_1, \ldots, i_d} a_{i_1, \ldots, i_d}^2\right)^{1/2}$ and is equivalent to the Frobenius norm of any of its unfoldings, i.e., $\|\Tn{X}\| = \|\Mx{X}_{(k)}\|_F$.

We use the symbol $\krn$ to denote the Kronecker product of vectors or matrices and $\krp$ for the KRP of matrices. For matrices $\Mx{A} \in \R^{m \times n}$ and $\Mx{B} \in \R^{p \times q}$, the Kronecker product $\Mx{A} \krn \Mx{B} \in \R^{mp \times np}$  is a block matrix defined as follows:
\[
\Mx{A} \krn \Mx{B}  = \begin{bmatrix}
    a_{11} \Mx{B}& \cdots & a_{1n} \Mx{B}\\
    \vdots & \ddots & \vdots\\
    a_{m1} \Mx{B} & \cdots & a_{mn} \Mx{B}
\end{bmatrix}
\]
The KRP is defined as a column-wise Kronecker product, i.e., if $\Mx{A} = \bmat{\Vc{a}_1 &  \dots & \Vc{a}_n}$ and $\Mx{B} = \bmat{\Vc{b}_1 & \dots &  \Vc{b}_n}$, then 
\[\Mx{A} \krp \Mx{B} = \bmat{\Vc{a}_1 \krn \Vc{b}_1 &  \dots & \Vc{a}_n \krn \Vc{b}_n} .
\]

\subsection{Tensor Train}\label{ssec:tt} A $d$-way tensor $\Tn{X} \in \R^{n_1 \times \cdots \times n_d}$ is said to be in tensor train (TT) format if there exist positive integers $r_0, \ldots, r_d$ called \textit{TT ranks}, with $r_0 = r_d=1$ and a sequence of $3$-way tensors $\Tn{X}_1,\ldots, \Tn{X}_d$ called \textit{TT cores} with $\Tn{X}_{k} \in \R^{r_{k-1} \times n_k \times r_k}$ for $1\leq k \leq d$ such that 
\begin{equation}
    \Tn{X}(i_1, \ldots, i_d) = \Tn{X}_1(i_1,:) \cdots \Tn{X}_k(:,i_{k-1},:) \cdots \Tn{X}_{d}(:,i_d)
\end{equation}
where $1\leq i_k \leq n_k$. We denote TT format tensor as $\Tn{X} = [\Tn{X}_1, \ldots, \Tn{X}_d]$. Since $r_0 = r_d = 1$, the first and last TT cores, $\Tn{X}_1$ and $\Tn{X}_d$ are matrices. $\Tn{X}_1(i_1, :)$ and $\Tn{X}_d(:, i_d)$ are  vectors  of dimension $r_1$ and $r_{d-1}$, respectively. For each core $\Tn{X}_k$, the slice $\Tn{X}_k(:, i_k, :)$ is an $r_{k-1} \times r_k$ matrix, which is called the $i_k$-th slice of the $k$-th TT core of the tensor $\Tn{X}$. Due to the multiplicative nature of such a format, the TT decomposition is not unique. Figure~\ref{fig:tt_tensor} illustrates the tensor network diagram of the TT tensor $\Tn{X}$, where the nodes represent the core tensors and the edges correspond to the modes, allowing for the contraction of connected nodes.

\begin{figure}[ht!] 
\centering
\scalebox{0.8}{\begin{tikzpicture}[scale=0.7]
    \node[draw, circle, minimum size=1.2cm] (T1) at (0, 0) {$\Tn{X}_1$};
    \node[draw, circle, minimum size=1.2cm] (T2) at (3, 0) {$\Tn{X}_2$};
    \node[draw, circle, minimum size=1.2cm] (Td_1) at (8, 0) {$\Tn{X}_{d-1}$};
    \node[draw, circle, minimum size=1.2cm] (Td) at (11, 0) {$\Tn{X}_d$};

    \draw[-] (T1.east) -- (T2.west) node[midway, above] {$r_1$};
    \draw[-] (T2.east) -- ++(1.2,0) node[midway, above] {$r_2$};  
    \draw[dotted] (T2.east) ++(1.6,0) -- ++(0.6,0); 
    \draw[-] (Td_1.west) ++(-1.2,0) -- ++(1.2,0) node[midway, above] {$r_{d-2}$}; 
    \draw[-] (Td_1.east) -- (Td.west) node[midway, above] {$r_{d-1}$};



    \draw[-] (T1.south) -- ++(0,-1)  node[near end, right] {$n_1$};

    \draw[-] (T2.south) -- ++(0,-1)  node[near end, right] {$n_2$};

    \draw[-] (Td_1.south) -- ++(0,-1)  node[near end, right] {$n_{d-1}$};

    \draw[-] (Td.south) -- ++(0,-1)  node[near end, right] {$n_{d}$};
\end{tikzpicture}} 
\caption{Tensor network diagram for a TT-tensor $\Tn{X} = [\Tn{X}_1, \ldots, \Tn{X}_d ]$.}
\label{fig:tt_tensor}
\end{figure}
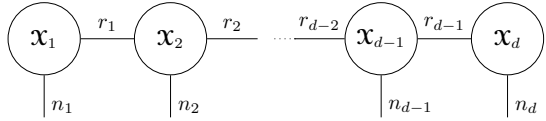

We use the following two matrix unfoldings of the $3$-way tensor to express linear algebra operations on TT cores: \textit{horizontal} and \textit{vertical} unfolding.  We use the operator $\h{\cdot}$ to denote the horizontal unfolding of a tensor. For a tensor $\Tn{X}_k$, $\h{\Tn{X}_k} \in \R^{r_{k-1} \times n_k r_k}$ represents the horizontal concatenation of the slices $\Tn{X}_k(:,i_k,:)$ for $i_k = 1, \dots, n_k$. Similarly, the vertical unfolding is obtained by concatenating the slices $\Tn{X}_k(:,i_k,:)$ for $i_k = 1, \dots, n_k$ vertically. We denote this operation by the operator $\v{\cdot}$, and $\v{\Tn{X}_k}$ results in a matrix of size $r_{k-1} n_k \times r_k$. 
An extension of the vertical and horizontal unfolding to the product of the TT cores is: $\v{\Tn{X}_{1:k}}$ and $\h{\Tn{X}_{k+1:d}}$, where $\Tn{X}{1:k}$ and $\Tn{X}{k+1:d}$ are the products of the left $k$ and right $(d-k)$ cores, respectively.
$\v{\Tn{X}_{1:k}} \in \R^{n_1 \cdots n_{k} \times r_k}$ corresponds to the transpose of the mode-$(k+1)$ unfolding of $\Tn{X}_{1:k}$, while $\h{\Tn{X}_{k+1:d}} \in \R^{r_k \times n_{k+1}\cdots n_d}$ is mode-$1$ unfolding of $\Tn{X}_{k+1:d}$.

The first $k$-mode unfolding $\Mx{X}_{(1:k)}$ of the  TT-tensor have rank $r_k$ representation
\begin{equation}
    \Mx{X}_{(1:k)} = \v{\Tn{X}_{1:k}} \h{\Tn{X}_{k+1:d}} 
\end{equation}
We can write the same unfolding as a product of four matrices as follows~\cite[Eq. (2.1)]{al2023randomized}:
\begin{equation}
\Mx{X}_{(1:k)} = \left(\Mx{I}_{n_k} \krn \v{\Tn{X}_{1:k-1}} \right) \v{\Tn{X}_k} \h{\Tn{X}_{k+1}} \left(\h{\Tn{X}_{k+1:d}} \krn \Mx{I}_{n_{k+1}} \right). \label{eq:unfold_1}
\end{equation}
Furthermore, we write it as a product of $(d-k+1)$ matrices as follows:
\begin{align}
    \Mx{X}_{(1:k)} &= \left(\Mx{I}_{n_k} \krn \v{\Tn{X}_{1:k-1}} \right) \v{\Tn{X}_k} \h{\Tn{X}_{k+1}}  \left(\h{\Tn{X}_{k+2}} \krn \Mx{I}_{n_{k+1}}\right) \cdots  \notag\\
    &\cdots
    \left(\h{\Tn{X}_{d-1}} \krn \Mx{I}_{n_{d-2}} \krn \cdots \krn \Mx{I}_{n_{k+1}}\right) \left(\h{\Tn{X}_{d}} \krn \Mx{I}_{n_{d-1}} \krn \cdots \krn  \Mx{I}_{n_{k+1}}\right).  \label{eq:unfold_2}
\end{align}
We use  these unfoldings in Section~\ref{sec:our_proposal}.
\subsection{TT-Rounding} \label{sec:tt_rounding}
First proposed in \cite{oseledets_tensor-train_2011}, TT rounding is the problem of computing a low-rank TT decomposition of an input already given in a higher-rank TT format. For the \textit{fixed-precision approximation problem}, given suboptimal ranks $r_k$, we want to estimate the true ranks $r'_{k} \leq r_{k}$ while maintaining the prescribed accuracy $\varepsilon$. In the \textit{fixed-rank} problem, given suboptimal ranks $r_k$ we seek an approximation with target ranks $\ell_k$.

To illustrate the TT rounding procedure, we first consider an analogy for matrices. Let $\Mx{Y} = \Mx{A} \Mx{B}^T$ where $\Mx{A}$ is an $m \times r$ matrix and $\Mx{B}$ is $r \times n$ with $r \leq \min\{m, n\}$. To obtain a low-rank approximation we could form the product $\Mx{A} \Mx{B}^T$ and then compute a truncated SVD, but exploiting the multiplicative structure gives us computational advantages. The computation of a rank $\ell < r$ approximation can be divided into an orthogonalization step followed by a compression phase. During the orthogonalization step, we want to make $\Mx{Y}$ right orthogonal by computing the thin QR factorization, $\Mx{B} = \Mx{Q} \Mx{R}$ and setting $\Mx{Z} = \Mx{A}\Mx{R}^T$. This initial step gives $\Mx{Y} = \Mx{A} \Mx{B}^T = \Mx{Z} \Mx{Q}^T$, where $\Mx{Q}^T$ has orthonormal rows. In the second step, $\Mx{Z}$ is compressed to rank $l$ via the SVD, $\Mx{Z} \approx \Mx{U} \Mx{\Sigma} \Mx{V}_{Z}^T$. The low-rank approximation of $\Mx{Y}$ is then given by $\Mx{Y} \approx \Mx{U} \Mx{\Sigma} (\Mx{Q}\Mx{V}_{Z})^T = \Mx{U} \Mx{\Sigma} \Mx{V}^T$. The classic TT rounding algorithm \cite{oseledets_tensor-train_2011} is also divided into two phases, orthogonalization followed by compression.

We say a tensor is \textit{right orthogonal} if all its horizontal unfoldings $\h{\Tn{X}_k}$ have orthonormal rows, except the first core. We similarly say that a tensor is \textit{left orthogonal} if all its vertical unfoldings $\v{\Tn{X}_k}$ have orthonormal columns except the last core.

\textit{Right-to-Left orthogonalization} is the procedure in which a TT tensor $\Tn{X}$ is transformed into an equivalent but right orthogonal TT tensor $\Tn{Y}$. Starting from the last core, we compute the thin {QR} factorization $\Mx{Q} \Mx{R} = \h{\Tn{X}_{d}}^T$ and set the cores $\Tn{Y}_{d-1}$ and $\Tn{Y}_{d}$ as

\begin{equation*}
    \v{\Tn{X}_{d-1}} \h{ \Tn{X}_{d}} = \underbrace{\v{\Tn{X}_{d-1}} \Mx{R}^T}_{\v{\Tn{Y}_{d-1}}} \underbrace{\Mx{Q}^T}_{\h{ \Tn{Y}_{d} }}.
\end{equation*}

This procedure is continued through the cores from $d-1$ to $2$. We could obtain instead a left orthogonal TT tensor similarly by processing the cores from left to right.

After obtaining a right-orthogonal tensor, we perform a compression step, processing the cores from left to right. For each mode, we compute a low-rank approximation of the vertical unfolding $\v{ \Tn{Y}_k }$ via the SVD. For the fixed precision problem, the number of singular vectors to keep depends on the threshold $\tau = \frac{\| \Tn{Y} \|}{\sqrt{d-1}} \varepsilon$. The vertical unfolding is set to $\v{ \Tn{Y}_k } = \Mx{U}$ while the other two factors $\Mx{\Sigma} \Mx{V}^T$ are combined with the horizontal unfolding of the next core $\h{\Tn{Y}_{k+1}}$ to obtain $\h{\Tn{Y}_{k+1}} = \Mx{\Sigma} \Mx{V}^T \h{\Tn{Y}_{k+1}}$. The process ends after $d-1$ steps. Details are given in Algorithm~\ref{alg:TT-Rounding}.

\begin{algorithm}[!ht]
  \caption{TT-Rounding \cite{oseledets_tensor-train_2011}}
  \label{alg:TT-Rounding}
  \begin{algorithmic}[1]
    \Require A TT-tensor $\Tn{X} = [\Tn{X}_1, \ldots, \Tn{X}_d]$ with TT-ranks $\{r_k\}$, user-defined threshold $\varepsilon>0$ or target ranks $\{\ell_k\}$
    \Ensure A TT-tensor $\Tn{Y} = [\Tn{Y}_1, \ldots, \Tn{Y}_d]$ with reduced ranks and desired accuracy or target ranks
    \Function{ $\Tn{Y}$ = TT-Rounding}{$\Tn{X}$, $\varepsilon$ or $\{\ell_k\}$}
    \State $\Tn{Y}$ = \textproc{OrthogonalizeRL}($\Tn{X}$) \Comment{right to left orthogonalization} \label{line:tt_rounding_orth_step}
    \State $\tau = \frac{\| \Tn{Y}_1 \|}{\sqrt{d-1}} \varepsilon$
    \For{$k = 1$ to $d-1$ }
    \State [ $\v{ \Tn{Y}_{k} }$, $\Mx{R}$ ] = \texttt{QR}($\v{\Tn{Y}_{k}}$) \Comment{reduced QR} decomposition
    \State [ ${\Mx{U}}$, ${\Mx{\Sigma}}$, ${\Mx{V}}^T$ ] = \texttt{SVD}($\Mx{R}$, $\tau$ or $\ell_{k}$ ) \Comment{$\tau$ or $\ell_{k}$ truncated SVD}
    \State $\v{ \Tn{Y}_{k} } = \v{ \Tn{Y}_{k} } {\Mx{U}}$
    \State $\h{ \Tn{Y}_{k + 1} } = {\Mx{\Sigma}} {\Mx{V}}^T \h{ \Tn{Y}_{k + 1} }$  \Comment{$\Tn{Y}_{k+1} = \Tn{Y}_{k+1} \times_1 ({\Mx{\Sigma}} {\Mx{V}}^T)$}
    \EndFor
    \EndFunction
\end{algorithmic}
\end{algorithm}

\subsection{Randomized Matrix Algorithms}
In this section, we discuss the randomized low-rank matrix approximation methods, particularly the randomized range finder and its adaptive version. 

\paragraph{Randomized Range Finder} We briefly review the randomized range finder~\cite[Algorithm 4.1]{halko2011finding}. Let $\Mx{X} \in \R^{m \times n}$ be a given matrix with $m \geq n$, and we want to compute a rank-$r$ approximation, where $r \le n$. To do so, we first chose an oversampling parameter $p$ and define $\ell = r + p$. We then generate a random matrix $\Mx{\Omega} \in \R^{n \times \ell}$ whose entries are independently and identically sampled from a standard normal distribution (with zero mean and unit variance). We refer to this matrix as a {Gaussian random matrix}. 

Using the matrix $\Mx{\Omega}$, we construct a sketch matrix $\Mx{S} = \Mx{X} \Mx{\Omega}$. We then compute the thin QR decomposition of $\Mx{S}$, yielding an orthonormal matrix $\Mx{Q} \in \R^{m \times \ell}$ whose columns form a basis for the range of $\Mx{S}$. If $\Mx{X}$ has rank close to $r$, or if its singular values decay rapidly after the $r$-th one, then the subspace spanned by the columns of $\Mx{Q}$ provides a good approximation to the column space of $\Mx{X}$. In this case, $\Mx{X} \approx \Mx{Q} \Mx{Q}^T \Mx{X}$ and we use $\Mx{Q} \Mx{Q}^T \Mx{X}$ as  a low-rank approximation of $\Mx{X}$. Assuming the matrix $\Mx{X}$ is stored in a dense format, the computational cost of finding $\Mx{Q}$ using this approach  is $2mn\ell + O(m\ell^2)$ floating point operations (flops).

This procedure yields a rank-$\ell$ approximation of $\Mx{X}$. We can convert it to desired rank-$r$ approximation by additional post-processing, refer to \cite{halko2011finding} for more details.

\paragraph{Adaptive Randomized Range Finder} We review the adaptive randomized range finder in~\cite{martinsson2020randomized}. Suppose we are given a matrix $\Mx{X} \in \R^{m \times n}$ and a tolerance $\epsilon$, and we want to find a low-rank approximation of $\Mx{X}$ that is accurate to within precision $\epsilon$. This formulation represents a more practical scenario compared to the fixed-rank setting, as in most applications we do not have prior knowledge of the singular value spectrum of $\Mx{X}$ that would allow us to determine an appropriate rank $r$ in advance. The idea of adaptive range finder is to start with a low initial guess for the rank,  run the range finder, and check if we are within the desired tolerance. If not,  then draw more samples to enrich the basis until the required accuracy is achieved. Algorithm~\ref{alg:adap_rand_range_finder} gives a pseudocode of the adaptive range finder procedure.
\begin{algorithm}[!ht]
  \caption{Adaptive Randomized Range Finder~\cite{martinsson2020randomized}}
  \label{alg:adap_rand_range_finder}
  \begin{algorithmic}[1]
    \Require Target matrix $\Mx{X} \in \R^{m \times n}$, tolerance $\varepsilon$ and block size $b$
    \Ensure $\|\Mx{X} - \Mx{Q} \Mx{Q}^T \Mx{X}||_F \leq \varepsilon$
    \Function{$\Mx{Q}$=RandRangeFinder}{$\Mx{X}$, $\varepsilon$, $b$}
    \State Select a Gaussian matrix $\Mx{\Omega} \in \R^{n \times b}$
    \State $\Mx{S} = \Mx{X} \Mx{\Omega}$ 
    \State $[\Mx{Q},\sim] = \texttt{QR}(\Mx{S})$ \Comment{reduced QR decomposition}
    \While {$\|\Mx{X} - \Mx{Q} \Mx{Q}^T \Mx{X}\|_F > \varepsilon$} \label{lin:error_AdpRangeFinder}
    \State Select a Gaussian matrix $\Mx{\Omega} \in \R^{n \times b}$
    \State $\Mx{S} = \Mx{X} \Mx{\Omega}$
    \State $\Mx{S} = \Mx{S} - \Mx{Q}  \Mx{Q}^T\Mx{S}$
    \State $[\Tilde{\Mx{Q}}, \sim] = \texttt{QR}(\Mx{S})$ \Comment{reduced QR decomposition}
    \State $\Mx{Q} = \bmat{\Mx{Q} & \Tilde{\Mx{Q}}}$ \Comment{add new directions to the basis}
    \EndWhile
    \EndFunction
\end{algorithmic}

\end{algorithm}

In Algorithm~\ref{alg:adap_rand_range_finder}, instead of computing the exact residual error in line~\ref{lin:error_AdpRangeFinder}, one can use a sketching-based estimate of its magnitude to reduce computational cost: 
\begin{equation} 
\|\Mx{X} - \Mx{Q} \Mx{Q}^T \Mx{X}\|_F^2 \approx \frac{1}{s} \|(\Mx{I} - \Mx{Q} \Mx{Q}^T) \Mx{X}\Mx{\Omega}\|_F^2, \end{equation}
where $\Mx{\Omega} \in \R^{n \times s}$ is a random matrix with independent entries drawn from the standard normal distribution. 
For more details, refer to \cite{martinsson2020randomized}.

\section{Randomized TT Rounding Algorithms using KRPs} \label{sec:our_proposal}
In this section, we propose fixed-rank (Algorithm~\ref{sec:fixed_rank}) and adaptive-rank (Algorithm~\ref{sec:adaptive_rank}) randomized TT-rounding algorithms that utilize a sequence of random matrices with a KRP structure for sketching. 

We give a high-level overview of both algorithms. The fixed-rank variant follows a procedure similar to that of the Randomize-then-Orthogonalize algorithm introduced by \cite{al2023randomized}, consisting of a randomization phase followed by an orthogonalization phase. The key distinction lies in the randomization step. In \cite{al2023randomized} a random Gaussian TT-tensor is employed to generate sketches, a TT-tensor whose cores $\Tn{R}_{k}\in \mathbb{R}^{l_{k-1} \times n_{k} \times l_{k}}$ are filled with independent, normally distributed entries with zero mean and variance $1/(\ell_{k-1} n_k \ell_k)$. This is further explained in Section~\ref{supp-sec:rto}. Our proposed approach utilizes random matrices with KRP structure for sketching, providing an alternative mechanism to compress the TT tensor efficiently. Our adaptive-rank algorithm applies the randomization followed by orthogonalization steps iteratively until the desired approximation tolerance is achieved. We discuss it in detail in Section~\ref{sec:adaptive_rank}. 
The randomization phase is a core component of both of our proposed algorithms. In this phase, we exploit the structural properties of the TT format along with the KRP structure of the random matrices to compute partial contractions efficiently. We begin with a discussion of this partial contraction (Section~\ref{ssec:partial}).

Throughout this section, let $\Tn{X}\in \R^{n_1 \times \cdots \times n_d}$ in TT format with TT ranks $\crly{r_k}_{k=1}^{d-1}$ and TT cores $\crly{\Tn{X}_k}_{k=1}^d$, so that we can write $\Tn{X} = [\Tn{X}_1,\dots,\Tn{X}_d]$. We want to find a compressed representation $\Tn{Y} = [\Tn{Y}_1,\dots,\Tn{Y}_d]$, with ranks $\{\ell_j\}_{j=1}^{d-1}$ and $\ell_j \le r_j $ for $1 \le j \le d-1$. For the analysis, it is convenient to assume that $n_1=\dots=n_d = n$ and the ranks are uniform, i.e., $r_1 = \dots = r_{d-1} = r$ and $\ell_1 = \dots = \ell_{d-1} = \ell$.

\subsection{Partial contractions}\label{ssec:partial} Given a sequence of matrices $\crly{\Mx{\Omega}_k}_{k \geq 2}^{d}$, where each $\Mx{\Omega}_k \in \R^{n_k \times \ell}$ for $2 \le k \le d$. We define the partial contraction matrices $\crly{\Mx{W}_k}_{k=2}^d$ as follows:
\begin{equation}\label{eqn:partialcontraction}
\begin{aligned}
    \Mx{W}_k \equiv & \> \h{\Tn{X}_{k:d}} \left( \Mx{\Omega}_d \krp \cdots \krp \Mx{\Omega}_k \right), \qquad 2 \le k < d \\
    \Mx{W}_{d} \equiv & \> \h{\Tn{X}_d} \Mx{\Omega}_d. 
\end{aligned}
\end{equation}
The matrices $\{\Mx{W}_k\}_{k=2}^d$ represent the sketches of the horizontal unfoldings, and they will play an important role in our subsequent algorithms.

It is clear how to compute $\Mx{W}_d$ from the definition. We can compute the remaining contraction matrices sequentially for \( k = d-1, \ldots, 2 \) using the following recurrence:
\begin{align}
    \Mx{W}_k 
    & = \h{\Tn{X}_k} \left( \h{\Tn{X}_{k{+}1}} \kron \Mx{I}_{n_k} \right) {\cdots} 
    \left( \h{\Tn{X}_d} {\kron} \Mx{I}_{n_{d{-}1}} {\kron} {\cdots} {\kron} \Mx{I}_{n_k} \right) 
    \left( \Mx{\Omega}_d \krp {\cdots} \krp \Mx{\Omega}_k \right)  \label{eq:partial_seq1} \\
    &= \h{\Tn{X}_k} \left[ \Mx{W}_{k+1} \krp \Mx{\Omega}_k \right], \label{eq:partial_seq2}
\end{align}
where $\Mx{I}{n_{j}}$ denotes the identity matrix of size $n_j \times n_j$, for $j \in \{k, \ldots, d-1\}$. Equation~\eqref{eq:partial_seq2} follows from \eqref{eq:partial_seq1} by the repeated application of the identity
$(\Mx{A} \kron \Mx{B})(\Mx{C} \krp \Mx{D}) = (\Mx{A} \Mx{C}) \krp (\Mx{B} \Mx{D})$. A pseudocode for computing the partial contraction with KRP-structured matrices, based on the above procedure, is provided in Algorithm~\ref{alg:partialContractionRL_KRP}. The algorithm allows computing only the sketches $\{\Mx{W}_k\}_{k=t}^d$ for some $ t \ge 2$. 

\begin{figure}[!ht] 
\centering
\resizebox{\textwidth}{!}{
\begin{tikzpicture}


\node[draw, circle, minimum size=1cm] (T1a) at (0,0) {$\Tn{X}_1$};
\node[draw, circle, minimum size=1cm] (T2a) [right=1.5cm of T1a] {$\Tn{X}_2$};
\node[draw, circle, minimum size=1cm] (T3a) [right=1.5cm of T2a] {$\Tn{X}_3$};
\node[draw, circle, minimum size=1cm] (T4a) [right=1.5cm of T3a] {$\Tn{X}_4$};

\draw[-] (T1a.east) -- (T2a.west) node[midway, above] {$r_1$};
\draw[-] (T2a.east) -- (T3a.west) node[midway, above] {$r_2$};
\draw[-] (T3a.east) -- (T4a.west) node[midway, above] {$r_3$};

\node[draw, circle, minimum size=1cm] (O2a) [below=1.5cm of T2a] {$\Mx{\Omega}_2$};
\node[draw, circle, minimum size=1cm] (O3a) [right=1.5cm of O2a] {$\Mx{\Omega}_3$};
\node[draw, circle, minimum size=1cm] (O4a) [right=1.5cm of O3a] {$\Mx{\Omega}_4$};

\draw[-] (T1a.south) -- ++(0,-1.2) node[near end, right] {$n_1$};
\draw[-] (T2a.south) -- (O2a.north) node[midway, right] {$n_2$};
\draw[-] (T3a.south) -- (O3a.north) node[midway, right] {$n_3$};
\draw[-] (T4a.south) -- (O4a.north) node[midway, right] {$n_4$};



\coordinate (circlecentera) at ($(T1a) + (1cm,-3.5cm)$);
\node[circle, draw, inner sep=1.2pt, fill=white] at (circlecentera) {};
\draw[-,shorten <=1pt] (circlecentera) -- ++(-0.6,0) node[left] {$\ell$};
\draw[-,shorten >= 1pt] (O2a.south) .. controls +(0,-0.4) and +(0,0) .. (circlecentera);
\draw[-,shorten >= 1pt] (O3a.south) .. controls +(0,-0.4) and +(0,0) .. (circlecentera);
\draw[-,shorten >= 1pt] (O4a.south) .. controls +(0,-0.7) and +(0,0) .. (circlecentera);

\node at ($(T2a)!0.5!(T3a) + (0,-4.5)$) {(a) Initial Structure};


\begin{scope}[xshift=10cm]

\node[draw, circle, minimum size=1cm] (T1b) at (0,0) {$\Tn{X}_1$};
\node[draw, circle, minimum size=1cm] (T2b) [right=1.5cm of T1b] {$\Tn{X}_2$};
\node[draw, circle, minimum size=1cm] (T3b) [right=1.5cm of T2b] {$\Tn{X}_3$};

\node[draw, circle, minimum size=1cm] (O2b) [below=1.5cm of T2b] {$\Mx{\Omega}_2$};
\node[draw, circle, minimum size=1cm] (O3b) [right=1.5cm of O2b] {$\Mx{\Omega}_3$};


\node[draw, circle, minimum size=1cm, fill=red!20] (W4b) at ($(T3b) + (-40:2.3cm)$) {$\Mx{W}_4$};

\draw[-] (T1b.east) -- (T2b.west) node[midway, above] {$r_1$};
\draw[-] (T2b.east) -- (T3b.west) node[midway, above] {$r_2$};
\draw[-] (T3b.east) -- (W4b.north) node[midway, above, sloped] {$r_3$};

Vertical connections
\draw[-] (T1b.south) -- ++(0,-1.2) node[near end, right] {$n_1$};
\draw[-] (T2b.south) -- (O2b.north) node[midway, right] {$n_2$};
\draw[-] (T3b.south) -- (O3b.north) node[midway, right] {$n_3$};




\coordinate (O4b) at ($(O3b) + (0.8cm, -0.87cm)$);
\draw[-] (O4b) -- (W4b.south);

\coordinate (O5b) at ($(O3b) + (-1.2cm, -0.95cm)$);
\node[circle, draw, inner sep=1.2pt, fill=white] at (O5b) {};

\coordinate (circlecenterb) at ($(T1b) + (1cm,-3.5cm)$);
\node[circle, draw, inner sep=1.2pt, fill=white] at (circlecenterb) {};
\draw[-,shorten <=1pt] (circlecenterb) -- ++(-0.6,0) node[left] {$\ell$};

\draw[-,shorten >= 1pt] (O2b.south) .. controls +(0,-0.4) and +(0,0) .. (circlecenterb);
\draw[-,shorten >= 1pt] (O3b.south) .. controls +(0,-0.38) and +(0,0) .. (O5b);
\draw[-,shorten <= 1pt] (O5b) -- (O4b);
\draw[-,shorten >=1pt, shorten <= 1pt] (circlecenterb) -- (O5b);

\node at ($(T2b)!0.5!(T3b) + (0,-4.5)$) {(b) Structure with $\Mx{W}_4$};

\end{scope}


\begin{scope}[xshift = 2cm, yshift=-6cm]

\node[draw, circle, minimum size=1cm] (T1c) at (0,0) {$\Tn{X}_1$};
\node[draw, circle, minimum size=1cm] (T2c) [right=1.5cm of T1c] {$\Tn{X}_2$};

\node[draw, circle, minimum size=1cm] (O2c) [below=1.5cm of T2c] {$\Mx{\Omega}_2$};

\node[draw, circle, minimum size=1cm, fill=red!20] (W3c) at ($(T2c) + (-40:2.3cm)$) {$\Mx{W}_3$};

\draw[-] (T1c.east) -- (T2c.west) node[midway, above] {$r_1$};
\draw[-] (T2c.east) -- (W3c.north) node[midway, above, sloped] {$r_2$};

\draw[-] (T1c.south) -- ++(0,-1.2) node[near end, right] {$n_1$};
\draw[-] (T2c.south) -- (O2c.north) node[midway, right] {$n_2$};

\coordinate (O3c) at ($(O2c) + (0.8cm, -1cm)$);
\draw[-] (O3c) -- (W3c.south);

\coordinate (circlecenterc) at ($(T1c) + (1cm,-3.5cm)$);
\node[circle, draw, inner sep=1.2pt, fill=white] at (circlecenterc) {};
\draw[-,shorten <=1pt] (circlecenterc) -- ++(-0.6,0) node[left] {$\ell$};
\draw[-,shorten >= 1pt] (O2c.south) .. controls +(0,-0.4) and +(0,0) .. (circlecenterc);
\draw[-,shorten <= 1pt] (circlecenterc) -- (O3c);

\node at ($(T1c)!0.6!(T2c) + (0,-4.5)$) {(c)Structure with $\Mx{W}_3$};

\end{scope}


\begin{scope}[yshift=-6cm, xshift = 11cm]

\node[draw, circle, minimum size=1cm] (T1d) at (0,-1.4) {$\Tn{X}_1$};
\node[draw, circle, minimum size=1cm, fill=red!20] (W2d) [right=1.5cm of T1d] {$\Mx{W}_2$};

\draw[-] (T1d.east) -- (W2d.west) node[midway, above] {$r_1$};

\draw[-] (T1d.south) -- ++(0,-1)  node[near end, right] {$n_{1}$};
\draw[-] (W2d.south) -- ++(0,-1)  node[near end, right] {$\ell$};

\node at ($(T1d)!0.6!(W2d) + (0,-3cm)$) {(d) Structure with $\Mx{W}_2$};

\end{scope}

\end{tikzpicture}
}
\caption{An illustration of partial contraction of tensor $\Tn{X}= [\Tn{X}_1, \Tn{X}_2, \Tn{X}_3, \Tn{X}_4 ]$ with matrix $\Mx{\Omega} = \Mx{\Omega}_4 \kr  \Mx{\Omega}_3 \kr \Mx{\Omega}_2$.  }
\label{fig:partialcontraction_krp}
\end{figure}
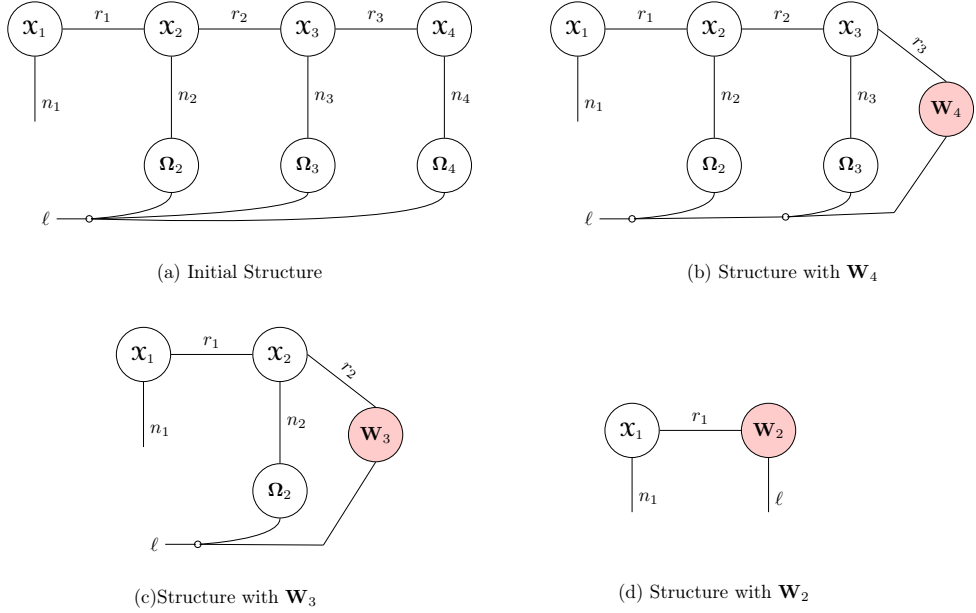
Figure~\ref{fig:partialcontraction_krp} provides a pictorial illustration of the right-to-left partial contraction process for a $4$-way TT tensor. We can compute the left-to-right partial contractions by changing the order of computation in Algorithm~\ref{alg:partialContractionRL_KRP} from right-to-left to left-to-right.

\begin{algorithm}[!ht]
  \caption{Right-to-Left Contraction with KRP Structured Matrix}
  \label{alg:partialContractionRL_KRP}
  \begin{algorithmic}[1]
    \Require A TT-tensor $\Tn{X} = [\Tn{X}_1, \ldots, \Tn{X}_d]$ with TT-ranks $\{r_k\}$ and random matrices $\Mx{\Omega}_t, \ldots, \Mx{\Omega}_d$ where $\Mx{\Omega}_k \in \Mx{R}^{n_k \times \ell}$ for $k \in \{t, \ldots, d\}$
    \Ensure Matrices $\mathbf{W}_{k} \in \mathbb{R}^{r_{k-1} \times \ell}$ for $k \in \{t, \ldots, d\}$ 
    \Function{[\{$\mathbf{W}_{k}$\}]=KRP-PartialContractionsRL}{$\Tn{X}$, $[\Mx{\Omega}_t, \ldots, \Mx{\Omega}_d]$}
    \State $\Mx{W}_{d}$ = $\h{\Tn{X}_{d}} \Mx{\Omega}_d$ \Comment{$\Mx{W}_{d}$ is $r_{d-1} \times \ell$ matrix} \label{lin:krp_last_mode_cont}
    \For{$k=d-1$ to $t$ }
    \State $\Mx{W}_{k}$ =$\h{\Tn{X}_{k}}\left[\Mx{W}_{k+1} \odot \Mx{\Omega}_{k}\right]$ \Comment{$\Mx{W}_{k}$ is $r_{k-1} \times \ell$ matrix} \label{lin:krp_squen_cont}
    \EndFor
    \EndFunction
\end{algorithmic}

\end{algorithm}

\paragraph{Computational cost} 
The computation of the partial contraction involves a sequence of matrix multiplications and KRP operations. In Algorithm~\ref{alg:partialContractionRL_KRP}, line~\ref{lin:krp_last_mode_cont} performs a matrix multiplication between matrices of size \(r \times n\) and \(n \times \ell\), which requires \(2nr \ell\) flops. Line~\ref{lin:krp_squen_cont} computes a matricized tensor times Khatri-Rao product (MTTKRP) involving matrices of sizes \(r \times \ell\) and \(n \times \ell\).
This can be done in $2nr^2 l + 2r^2l$ operations. 
See \cite[Section 3.5]{ballard2025tensor} for more details.
We denote the overall leading-order computational cost of Algorithm~\ref{alg:partialContractionRL_KRP} by $C_{par}^{(t)}$ and is:

\begin{equation} \label{eq:part_cont}
C_{\rm PartCont}^{(t)} =\sum_{k=t}^{d-1} \left(2 n r^2 \ell \right) = 2(d-t) n r^2 \ell, \quad \text{for } t \geq 2.
\end{equation}

\subsection{Fixed Rank Algorithm}~\label{sec:fixed_rank}
In this algorithm, we assume that the target ranks $\{\ell_k\}_{k=1}^{d-1}$ are known a priori. 

As in the Randomize-then-Orthogonalize algorithm (Algorithm~\ref{alg:randomize_then_orthogonalize}), our proposed algorithm has two phases: a randomization phase, followed by an orthogonalization phase. In the randomization phase,  we generate independent Gaussian random  matrices $\Mx{\Omega}_k \in \mathbb{R}^{n_k \times \ell}$ for $2 \le k \le d$, where $\ell = \max_{1 \le k \le d-1}\ell_k$. Next, we perform an efficient contraction between the KRP of these random matrices and the TT-tensor \(\Tn{X}\), as described in Algorithm~\ref{alg:partialContractionRL_KRP}, to compute the partial contraction matrices $\{\Mx{W}_k\}_{k=2}^d$. These steps constitute the randomization phase of the algorithm.

In the orthogonalization phase, we construct a left-orthogonal compressed TT-tensor $\Tn{Y}$, and it is similar to the orthogonalization phase of Algorithm~\ref{alg:randomize_then_orthogonalize}. In Step 1, $k=1$ and consider the unfolding $\Mx{X}_{(1)} = \v{\Tn{X}_{1}} \h{\Tn{X}_{2}}(\h{\Tn{X}_{3:d}} \kron \Mx{I})$ where $\Mx{I}$ is an identity matrix of appropriate size.
We compute the sketch \[\begin{aligned}\Mx{S}_1 =& \>  \Mx{X}_{(1)} \left( \Mx{\Omega}_d(:,1:\ell_1) \krp \cdots \krp \Mx{\Omega}_{2} (:,1:\ell_1)\right) \\= & \>   \v{\Tn{X}_1}\Mx{W}_2(:,1:\ell_1). \end{aligned}\]
We compute a thin-QR factorization $\Mx{S}_1 = \Mx{Q}_1\Mx{R}_1$, and compute an approximation  $$\Mx{X}_{(1)} \approx \underbrace{\Mx{Q}_1}_{\v{\Tn{Y}_1}} \underbrace{\Mx{Q}_1^T\v{\Tn{X}_1} \h{\Tn{X}_{2}}}_{\h{\Tn{Y}_2}}(\h{\Tn{X}_{3:d}} \kron \Mx{I}).$$ 
We set $\v{\Tn{Y}_1} = \Mx{Q}_1$ and $\h{\Tn{Y}_2} = (\Mx{Q}_1^T\v{\Tn{X}_1}) \h{\Tn{X}_2}$. Proceeding iteratively, we have the approximation
\[ \Mx{X}_{(1:k)} \approx (\Mx{I} \kron \v{\Tn{Y}_{1:k-1}})\v{\Tn{Y}_k} \h{\Tn{X}_{k+1}} (\h{\Tn{X}_{k+2:d}} \kron \Mx{I})   , \quad 2 \le k \le d-1.\]
 At the $k$-th step, since the first $k - 1$ cores of $\Tn{Y}$ have already been orthogonalized, the matrix $(\Mx{I} \kron \v{\Tn{Y}_{1:k-1}})$ has orthonormal columns and does not need to be further compressed.
We compute the sketch $\Mx{S}_k$ and its thin-QR factorization as 
\[\begin{aligned}
    \Mx{S}_k  =& \> \v{\Tn{Y}_k} \h{\Tn{X}_{k+1}}(\left(\h{\Tn{X}_{k+2:d}} \kron\Mx{I}) \left( \Mx{\Omega}_d(:,1:\ell_k) \krp \cdots \krp \Mx{\Omega}_{k+1} (:,1:\ell_k)\right)\right)\\
    = & \> \v{\Tn{Y}_k} \Mx{W}_{k+1}(:,1:\ell_k)= \Mx{Q}_k \Mx{R}_k. 
\end{aligned}\]
Thus, we have the approximation for a fixed $k$, 
\[ \Mx{X}_{(1:k)} \approx (\Mx{I} \kron \v{\Tn{Y}_{1:k-1}}) \underbrace{\Mx{Q}_k}_{\v{\Tn{Y}_k}} \underbrace{\Mx{Q}_k^T \v{\Tn{Y}_k} \h{\Tn{X}_{k+1}}}_{\h{\Tn{Y}_{k+1}}}(\h{\Tn{X}_{k+2:d}} \kron \Mx{I})   .\] 

The TT-cores are then updated by overwriting \(\v{\Tn{Y}_k} \leftarrow \Mx{Q}_k\), and computing \(\h{\Tn{Y}_{k+1}} = (\Mx{Q}_k^T \v{\Tn{Y}_k})  \h{\Tn{X}_{k+1}}\). A pseudo-code of our fixed rank algorithm is given in Algorithm~\ref{alg:krp_ttroudning_fixRank}.

 \begin{algorithm}
    \caption{Fixed Rank TT-Rounding: Randomize-then-Orthogonalize using KRP}
     \label{alg:krp_ttroudning_fixRank}
  \begin{algorithmic}[1]
    \Require TT-tensors $\Tn{X} = [\Tn{X}_1, \ldots, \Tn{X}_d]$ with TT-ranks $\{r_k\}$, target ranks $\{\ell_k\}$
    \Ensure A tensor $\Tn{Y}= [\Tn{Y}_1, \ldots, \Tn{Y}_d]$  with TT-ranks $\{\ell_k\}$
    \Function{$\Tn{Y}$=TT-Rounding-RandOrth-KRP}{$\Tn{X}$, $\{\ell_k\}$}
    \State Select random Gaussian matrices $\Mx{\Omega}_2, \ldots, \Mx{\Omega}_d$ of size $n_k \times \ell$ for $2 \leq k \leq d$ where $\ell = \max\{\ell_k\}$
    \State $\{\Mx{W}_k\}$ = \textproc{KRP-PartialContractionsRL}($\Tn{X}$, $[\Mx{\Omega}_2, \ldots, \Mx{\Omega}_d]$) \label{lin:partialcont_krp}
    \State $\Tn{Y}_{1}$ = $\Tn{X}_{1}$
    \For{$k=1$ to $d-1$ }
     \State $\Mx{S}_k$ = $\v{\Tn{Y}_{k}} \Mx{W}_{k+1}(:,1:\ell_k)$ \Comment{form sketched matrix $n_k \ell_{k-1} \times \ell_k$} \label{lin:sketch_comp_krp}
    \State [$\Mx{Q}_k$, $\sim$] = \texttt{QR}($\Mx{S}_k$) \Comment{reduced QR decomposition} \label{lin:thin_qr_krp}
    \State $\Mx{M}_k$ = $\Mx{Q}_k^T\v{\Tn{Y}_{k}}$  \Comment{form $\ell_k \times r_k$ matrix}  \label{lin:M_k_kpr}
    \State Set $\v{\Tn{Y}_k}=\Mx{Q}_k$  and $\h{\Tn{Y}_{k+1}} = \Mx{M}_k \h{\Tn{X}_{k+1}}$ \Comment{$\Tn{Y}_{k+1} = \Tn{X}_{k+1} \times_1 \Mx{M}_k$} \label{lin:M_to_next_core_kpr}
    \EndFor
    \EndFunction
\end{algorithmic}
\end{algorithm}

\textbf{Computational cost:} Line~\ref{lin:partialcont_krp} invokes the right-to-left partial contraction Algorithm~\ref{alg:partialContractionRL_KRP} with $t = 2$, resulting in a computational cost of  $C_{\rm PartCont}^{(2)}$.  Line~\ref{lin:sketch_comp_krp} computes the final sketch matrix and involves matrix multiplication, resulting in $2 n \ell r$ operations for $k=1$ and $2 n \ell^2 r$ operations for $k \geq 2$.  Line~\ref{lin:thin_qr_krp} performs a thin QR decomposition on a tall-skinny matrix of size $n\ell  \times \ell$, which costs $ 4n\ell^3 + O(\ell^3)$ flops. Line~\ref{lin:M_k_kpr} computes $\Mx{M}_k$ by multiplying matrices of size $\ell \times n\ell $ and $n\ell  \times r$, resulting in \( 2 n \ell^2 r \) flops.  Finally, Line~\ref{lin:M_to_next_core_kpr} prepares the core \( \Tn{Y}_{k+1} \) for the next iteration and involves matrix multiplication, costs $2 n \ell r^2$ flops for $k < d-1$ and $2 n \ell r$ for $k=d-1$. The total computational cost  of Algorithm~\ref{alg:krp_ttroudning_fixRank} in flops is 
\begin{align*}
C_{\rm FixRank} = C_{\rm PartCont}^{(2)} + 2(d-2) n \ell r^2 + 4(d-2) n \ell^2 r + 4(d-1) n \ell^3  + 4n \ell r +  O(d\ell^3).
\end{align*}
To leading order, this cost is $4dnr^2\ell$ flops.
Table~\ref{tab:rounding_costs} summarizes the comparison of the leading-order computational and the number of random numbers generated  (denoted RNG) costs of different algorithms for rounding. 

\begin{table}[!ht]
\caption{Leading-order computational and random number generation (RNG) costs for TT rounding algorithms.}
    \label{tab:rounding_costs}
    \centering
    \begin{tabular}{l|ccccc}
    \hline
        \textbf{Algo.} & TT-Rounding & Orth-Rand & Rand-Orth &  Rand-Orth-KRP \\
        \hline
        Cost & $O(dnr^3)$&$O(dnr^3)$ & $O(dnr^2 \ell)$ &$O(dnr^2 \ell)$  \\
        \hline 
         RNGs & - & $O(dr \ell)$ & $O(dn \ell^2)$ & $O(d n\ell)$ \\ 
        \hline
    \end{tabular}
    
\end{table}

\subsection{Adaptive Rank Algorithm}~\label{sec:adaptive_rank}
We want to find a compressed TT-tensor $\Tn{Y}$  while maintaining the prescribed relative error $\varepsilon$, i.e., $\|\Tn{X} - \Tn{Y}\| \leq \varepsilon\| \Tn{X}\|$. Our adaptive rank algorithm, Algorithm~\ref{alg:krp_ttroudning_adapRank}, is motivated by the adaptive randomized range finder, Algorithm~\ref{alg:adap_rand_range_finder}. It incrementally builds the core tensors $\{\Tn{Y}_i\}_{i=1}^d$ of the left-orthogonal compressed TT-tensor $\Tn{Y}$ by combining the randomized range finder and a blocked Gram-Schmidt orthogonalization scheme. We use an estimate of the norm of the residual as the stopping criterion for terminating the iterations in the range finder.

The algorithm requires user-defined parameters $0 < f_{\rm init} < 1$, $ 0 < f_{\rm inc} < 1$ representing fractions of the input ranks to be used as an initial sketching sizes and incremental sketching sizes, respectively. 
In numerical experiments, we use $f_{\rm init} = 0.1$ and $f_{\rm inc}=0.05$, but these can be tuned in applications. 
First, we generate independent Gaussian random  matrices $\Mx{\Omega}_k \in \mathbb{R}^{n_k \times r}$ for $2 \le k \le d$, where $r = \lceil\max\{r_k\} \cdot f_{\rm init} \rceil$. 
Then, using these random matrices, we compute the partial contraction matrices $\{\Mx{W}_k\}_{k=2}^d$ by invoking Algorithm~\ref{alg:partialContractionRL_KRP}.  
We  evenly distribute the required accuracy $\varepsilon$ across the TT-cores $\{\Tn{X}_k\}_{k=1}^{d-1}$ and denote the resulting modewise threshold by $\tau = \varepsilon \|\Tn{X}\|/\sqrt{d-1}$. A justification for this allocation is given in \cite[Theorem 8.6]{ballard2025tensor}. We begin with $\Tn{Y}_1 = \Tn{X}_1$.  The process of computing the left orthonormal compressed TT-core $\Tn{Y}_k$ for mode $k$ can be divided into two phases for each $k$, where $1 \le k \le d-1$. 

\textbf{Phase 1: {Initialization}} This phase consists of computing the $b_{\rm init}^{(k)}$ orthonormal basis  
\begin{equation*}\label{eqn:binit}b_{\rm init}^{(k)} = \lceil f_{\rm init} \cdot \bar{r}_k \rceil,\end{equation*}
where $\bar{r}_k$ denotes the minimum of the number of rows and columns in the updated core $\v{\Tn{Y}_k}$. 
The first phase of the algorithm to initialize the low-rank factors essentially follows the fixed rank case (Algorithm~\ref{alg:partialContractionRL_KRP}) with $\ell_k = b_{\rm init}^{(k)}$. 
At line~\ref{line:adp_init_sketch}, we compute the sketch matrix $\Mx{S}_k$ using partial contraction $\Mx{W}_{k+1}$ as 
\begin{align}
\Mx{S}_k &=\v{\Tn{Y}_{k}} \underbrace{\h{\Tn{X}_{k+1:d}} (\Mx{\Omega}_{d}(:,1:b_{\rm init}^{(k)}) \kr \cdots \kr \Mx{\Omega}_{k+1}(:,1:b_{\rm init}^{(k)}) )}_{\Mx{W}_{k+1}(:, 1:b_{\rm init}^{(k)})} \notag\\
&= \v{\Tn{Y}_{k}} \Mx{W}_{k+1}(:, 1:b_{\rm init}^{(k)}). \notag 
\end{align}
The QR decomposition of the sketch matrix $\Mx{S}_k$
yields the orthonormal basis of dimension $b_{\rm init}^{(k)}$, which we denote by $\Mx{Q}_k$ and set $\h{\Tn{Y}_{k+1}} = \Mx{Q}_k^T \v{\Tn{Y}_k} \h{\Tn{X}_{k+1}}$.

\textbf{Phase 2: {Iteration}} Let $b_{\text{inc}}^{(k)}$ be the incremental block size for $1\le k\le d-1$. We define it as 
\[
b_{\rm inc}^{(k)} = \lceil f_{\text{inc}} \cdot \bar{r}_k\rceil,
\]
where $\bar{r}_k$ is the minimum of the number of rows and columns in the core $\v{\Tn{Y}_k}$.
To expand the orthonormal basis, we compute the residual sketch $\Mx{S}_k$ and update the partial contraction matrices $\{\Mx{W}_j\}_{j = k+1}^d$ by invoking Algorithm~\ref{alg:resdSketching}. We explain the working of Algorithm~\ref{alg:resdSketching} as follows: if the partial contraction matrix $\Mx{W}_{k+1}$ does not contain $b_{\rm inc}^{(k)}$ additional columns (beyond those already used), we generate new random Gaussian matrices $\{\Mx{\Omega}_j\}_{j=k+1}^d$ and  construct new partial contractions $\{\Mx{W}_j^{\text{new}}\}_{j=k+1}^d$ using Algorithm~\ref{alg:partialContractionRL_KRP}, and concatenate them with the existing contraction matrices:
\[
\Mx{W}_j \leftarrow \bmat{\Mx{W}_j & \Mx{W}_j^{\text{new}}}, \quad \text{for all } j = k+1, \ldots, d.
\]
We then form the next sketch matrix:
\begin{equation}
\Mx{S}_k \leftarrow \v{\Tn{Y}_k} \Mx{W}_{k+1}(:, \col{\Mx{Q}_k}+1 : \col{\Mx{Q}_k}+b_{\rm inc}^{(k)}),  \label{eq:resid_sketch_a}
\end{equation}
where $\col{\cdot}$ denotes the current number of columns in the matrix, and orthogonalize it against the existing basis to get a residual sketch:
\begin{equation}
\Mx{S}_k \leftarrow \Mx{S}_k - \Mx{Q}_k \left( \Mx{Q}_k^T \Mx{S}_k \right).  \label{eq:resid_sketch_b}
\end{equation}
The quantity $\|\Mx{S}_k\|_F/\sqrt{b_{\rm inc}^{(k)}}$   provides an unbiased estimate of the Frobenius norm of the residual error matrix. 
A detailed analysis of this estimate is provided in Section~\ref{ssec:analysis}. 
If the estimate exceeds $\tau$, we compute the QR decomposition of the residual sketch to get new orthonormal directions 
\[
\Mx{S}_k = \Mx{Q}_k^{new} \, \Mx{R}_k^{new}.
\]
We reorthogonalize the directions in $\Mx{Q}_k^{new}$ with respect to the current basis $\Mx{Q}_k$:
\[
\Mx{Q}_k^{new} \leftarrow \texttt{QR}\left(\Mx{Q}_k^{new} - \Mx{Q}_k \left(\Mx{Q}_k^T \Mx{Q}_k^{new} \right) \right).
\]
This helps to mitigate numerical errors and ensures that the updated basis $\Mx{Q}_k \leftarrow \bmat{\Mx{Q}_k & \Mx{Q}_k^{new}}$ remains numerically orthonormal~\cite[Section 4.3]{martinsson2016randomized}. 
We then efficiently update the next core by utilizing the $\h{\Tn{Y}_{k+1}}$ computed so far in the previous steps as follows:
\[
\h{\Tn{Y}_{k+1}} \leftarrow \Mx{Q}_k^T \v{\Tn{Y}_k} \h{\Tn{X}_{k+1}} =  \bmat{\h{\Tn{Y}_{k+1}}\\ \left(\Mx{Q}_k^{new}\right)^T \v{\Tn{Y}_k} \h{\Tn{X}_{k+1}}}.
\]
We repeat step 2 until the residual norm falls below the threshold $\tau$ or the maximum possible rank $\bar{r}_k$ (i.e., minimum of the number of rows and columns in $\v{\Tn{Y}_k}$) is reached and set $\v{\Tn{Y}_k} = \Mx{Q}_k$.

See Algorithm~\ref{alg:krp_ttroudning_adapRank} for a detailed pseudocode of the above procedure.
We give a few remarks regarding Algorithm~\ref{alg:krp_ttroudning_adapRank}.

\begin{algorithm}[!ht]
    \caption{Residual sketching for basis expansion using KRPs.}
     \label{alg:resdSketching}
    \begin{algorithmic}[1]
    \Require TT tensor $\Tn{X}$,  matrix $\Mx{Z} \in \R^{\ell_{k-1} n_k \times r_k}$, basis matrix $\Mx{Q}$, partial contractions $\{\Mx{W}_{j}\}_{j={k+1}}^d$,  TT core index $k$ and block size $b$.
    \Ensure Residual sketch $\Mx{S}$ and partial contraction matrices $\{\Mx{W}_j\}_{j=k+1}^d$ appended with new columns 
    \Function{[$\Mx{S}$, $\{\Mx{W}_j\}$]=GenerateResidualSketch}{$\Tn{X}$, $\Mx{Z}$, $\Mx{Q}$, $\{\Mx{W}_{j}\}_{j=k+1}^d$, $k$, $b$}
      \If {$\col{\Mx{W}_{k+1}} < \col{\Mx{Q}} + b$} \Comment{$\col{\cdot}$ returns  number of columns} 
        \State Select Gaussian random matrices $\Mx{\Omega}_{k+1}, \ldots, \Mx{\Omega}_d$ of size $n_j \times (b - \col{\Mx{W}_{k+1}})$ for $k+1\leq j \leq d$ 
        \State $\{\Mx{W}_j^{new}\}$ = \textproc{KRP-PartialContractionsRL}($\Tn{X}$, $[\Mx{\Omega}_{k+1}, \ldots, \Mx{\Omega}_d]$)
            \For {$j = k+1$ to $d$}
                \State $\Mx{W}_{j} = \bmat{\Mx{W}_j & \Mx{W}_{j}^{new}}$ \Comment{append columns}
            \EndFor
        \EndIf
        \State  $\Mx{S}$ = $\Mx{Z} \Mx{W}_{k+1}\left(:,\col{\Mx{Q}}+1: \col{\Mx{Q}} +b \right)$
        \State $\Mx{S} = \Mx{S} - \Mx{Q} \left(\Mx{Q}^T \Mx{S} \right)$  \label{line:ResSketch_orth} \Comment{orthogonalization of $\Mx{S}$}
    \EndFunction
\end{algorithmic}

\end{algorithm}

\begin{algorithm}[!ht]
    \caption{Adaptive Rank TT-Rounding: Randomize-then-Orthogonalize using KRP}
     \label{alg:krp_ttroudning_adapRank}
  \begin{algorithmic}[1]
   \small
    \Require TT-tensors $\Tn{X} = [\Tn{X}_1, \ldots, \Tn{X}_d]$ with TT-ranks $\{r_k\}$, norm of the tensor $\Tn{X}$ \texttt{nrmx}, tolerance $\varepsilon$,  $0 < f_{\rm init} < 1$ to set initial block size, $0 < f_{\rm inc} < 1$ to set increment block size
    \Ensure A tensor $\Tn{Y}= [\Tn{Y}_1, \ldots, \Tn{Y}_d]$  with $\|\Tn{X} - \Tn{Y}\|/\|\Tn{X}\| \leq \varepsilon$ with high probability.
   
    \Function{$\Tn{Y}$=TT-Rounding-RandOrth-KRP-Adaptive}{$\Tn{X}$, \texttt{nrmx}, $\varepsilon$, $f_{\rm init}$, $f_{\rm inc}$}
    \State Select random Gaussian matrices $\Mx{\Omega}_2, \ldots, \Mx{\Omega}_d$ of size $n_k \times r$ for $2 \leq k \leq d$ where $r = \lceil\max(\{r_j\}_{j=1}^{d-1}) \cdot f_{\rm init} \rceil$
    \State $\{\Mx{W}_k\}$ = \textproc{KRP-PartialContractionsRL}($\Tn{X}$, $[\Mx{\Omega}_2, \ldots, \Mx{\Omega}_d]$)  \label{line:adp_init_part_con}
    \State $\tau = (\varepsilon \cdot \texttt{nrmx})/\sqrt{d-1}$
    \State $\Tn{Y}_1$ = $\Tn{X}_1$
    \For{$k=1$ to $d-1$ }
        \State $b_{init}^{(k)}  = \lceil \bar{r}_k \cdot f_{init} \rceil$  \Comment{initial block size, where $\bar{r}_k = \min(\size{\v{\Tn{Y}_k}})$ }
       \State $\Mx{S}_k = \v{\Tn{Y}_k} \Mx{W}_{k+1}(:,1:b_{init}^{(k)})$ \Comment{form sketched matrix $n_k \ell_{k-1} \times b_{init}^{(k)}$}  \label{line:adp_init_sketch}
        \State $[\Mx{Q}_k,\sim]$ = $\texttt{QR}(\Mx{S}_k)$ \Comment{reduced QR decomposition}
        \State $\Mx{M}_k = \Mx{Q}_k^{T} \v{\Tn{Y}_k}$ \Comment{form $b_{init}^{(k)} \times r_k$ matrix} 
        \State $\h{\Tn{Y}_{k+1}} = \Mx{M}_k \h{\Tn{X}_{k+1}}$ \Comment{$\Tn{Y}_{k+1} = 
        \Tn{X}_{k+1} \times_1 \Mx{M}_k$} \label{line:adp_step1_end}
        \State $b_{inc}^{(k)} = \lceil \bar{r}_k \cdot f_{inc} \rceil$  \Comment{incremental block size}
        \State [$\Mx{S}_k$, $\{\Mx{W}_j\}$] = \textproc{GenerateResidualSketch}($\Tn{X}$, $\v{\Tn{Y}_k}$, $\Mx{Q}_k$, $\{\Mx{W}_j\}$, $k$, $b_{inc}^{(k)}$)
        \While {$\|\Mx{S}_{k}\|_F/\sqrt{b_{inc}^{(k)}} \, > \, \tau$} 
                \State $[\Mx{Q}_k^{new},\sim] = \texttt{QR}(\Mx{S}_k)$  \Comment{reduced QR decomposition}
                \State $\Mx{Q}_k^{new} = \texttt{QR}\left(\Mx{Q}_k^{new} - \Mx{Q}_k \left(\Mx{Q}^T_k \Mx{Q}_k^{new} \right) \right)$ \label{line:adap_reOrtho}\Comment{re-orthogonalization }
                \State $\Mx{Q}_k= \bmat{\Mx{Q}_k & \Mx{Q}_k^{new}}$ \Comment{add new orthonormal directions}
                \State $\Mx{M}_k = \left(\Mx{Q}_k^{new} \right)^T \v{\Tn{Y}_k}$
                \State $\h{\Tn{Y}_{k+1}} = \bmat{\h{\Tn{Y}_{k+1}} \\ \Mx{M}_k \,\h{\Tn{X}_{k+1}}}$ 
        \State [$\Mx{S}_k$, $\{\Mx{W}_j\}$] = \textproc{GenerateResidualSketch}($\Tn{X}$, $\v{\Tn{Y}_k}$, $\Mx{Q}_k$, $\{\Mx{W}_j\}$, $k$, $b_{inc}^{(k)}$)
        \EndWhile
        \State $\v{\Tn{Y}_k} = \Mx{Q}_k$
    \EndFor
    \EndFunction
\end{algorithmic}
\end{algorithm}

\begin{remark}[Input tensor norm estimation]\label{rem:normestimation}
    In Algorithm~\ref{alg:krp_ttroudning_adapRank}, we assume that the norm of the input tensor $\Tn{X}$ is known. 
    If this is not the case, computing it exactly via orthogonalization or TT inner product requires as much computation as deterministic TT rounding \cite{ABB22}.
    Instead, we can estimate it using the partial contraction computed in line~\ref{line:adp_init_part_con} as follows:
    \begin{equation}
        \|\Tn{X}\| \approx \frac{\|\v{\Tn{X}_1}\Mx{W}_{2}\|_F}{\sqrt{r}}. \label{eq:resid_norm_estimation}
    \end{equation} 
    where $r$ denotes the sketch size (number of columns in $\Mx{W}_2$). An analysis of the norm estimation as well as the residual error estimation in the While loop of Algorithm~\ref{alg:krp_ttroudning_adapRank} is given in Section~\ref{ssec:analysis}. 
\end{remark}
\begin{remark}[Additional Compression] \label{rem:rounding_pass}
   For $k \in \{1, \ldots, d-1\}$, Algorithm~\ref{alg:krp_ttroudning_adapRank} incrementally enriches the initially computed orthogonal basis by a block of size $b_{\text{inc}}^{(k)}$ until the residual norm falls below the threshold $\tau$. There is the possibility that the ranks of the compressed tensor produced by this algorithm may exceed those obtained from the deterministic TT-rounding algorithm (Algorithm~\ref{alg:TT-Rounding}) due to either a large incremental block size or an inaccurate estimate of the residual norm.
The cores of the compressed tensor produced by Algorithm~\ref{alg:krp_ttroudning_adapRank} are left-orthonormal (except the last core).  
We can reduce the ranks by performing a compression pass on the tensor cores, similar to that in the deterministic TT-rounding algorithm (Algorithm~\ref{alg:TT-Rounding}, Lines~\ref{line:comp_start}--\ref{line:comp_end}), but in the opposite direction, i.e., for $k = d$ down to $2$, with truncation threshold $\tau$.
\end{remark}

\begin{remark}[Sum of TT-tensors] \label{rem:sum}
The adaptive rank algorithm (Algorithm~\ref{alg:krp_ttroudning_adapRank}) can be tailored to inputs that are sums of TT tensors.
One approach to handling such inputs is to compute a formal sum of the TT tensors, resulting in a TT tensor whose ranks are the sums of the TT summand ranks but whose cores have sparse structure, and then apply a rounding algorithm.
Instead, we can exploit the property of our randomized algorithm that the sketch of the sum is the sum of the sketches.
That is, we can sketch each TT summand individually using the approach of Algorithm~\ref{alg:partialContractionRL_KRP} and then accumulate the result efficiently.
We omit the details of this case because it follows the approach described in \cite[Section 3.3]{al2023randomized}; the only difference for the fixed-rank case is in how the partial contractions are computed to exploit the KRP structure of the sketches.
We also incorporate error estimation and basis augmentation to enable adaptivity, as in the case of a single TT tensor input.
We demonstrate the efficiency of this approach in Algorithm~\ref{ssec:cookies}.
\end{remark}

In the adaptive-rank algorithm (Algorithm~\ref{alg:krp_ttroudning_adapRank}), the extra computational overhead compared to the fixed-rank algorithm (Algorithm~\ref{alg:resdSketching}) arises from the orthogonalization step (Line~\ref{line:ResSketch_orth} of Algorithm~\ref{alg:resdSketching}) and the re-orthogonalization step (Line~\ref{line:adap_reOrtho} of Algorithm~\ref{alg:krp_ttroudning_adapRank}). However, both of these steps involve operations on smaller matrices and introduce only a minor overhead. As a result, the leading-order computational cost remains the same for both the fixed-rank and adaptive-rank algorithms.

\subsection{Analysis of norm estimate}\label{ssec:analysis} Here, we analyze the accuracy of estimating the Frobenius norm of a matrix using randomized sketches based on KRPs of Gaussian matrices. In the following theorem, we set $N = \prod_{k=1}^d n_k$.

\begin{theorem}\label{thm:frobenius-preservation} Let $\Mx{A} \in \R^{m \times N}$  and $\Mx{\Omega} = \Mx{\Omega}_d \krp \cdots \krp \Mx{\Omega}_2\krp \Mx{\Omega}_1$, where each $\Mx{\Omega}_k \in \R^{n_k \times \ell}$ is Gaussian random matrix and $d > 1$. Then $\frac{1}{\ell} \|\Mx{A} \Mx{\Omega}\|_F^2$ gives an unbiased estimate of the $\|\Mx{A}\|_F^2$. Let  $\varepsilon >0$  and $ 0 < \delta \leq 1$. If the number of samples $\ell$ satisfies
$$\ell \geq 4\max\big(\varepsilon^{-1} C_\alpha'' \ln^d(2/\delta ), \, \varepsilon^{-2} {C_{\alpha}'}^2 \ln^2(2/\delta)\big),$$ 
where $C_{\alpha}'$ and $C_{\alpha}''$ are constants depending only on $\alpha \equiv 1/d$,  then with probability at least $1-\delta$
\[  \left|\frac{1}{\ell}\|\Mx{A} \Mx{\Omega}\|_F^2 - \|\Mx{A}\|_F^2 \right| \leq \varepsilon\|\Mx{A}\|_F^2.
\]
\end{theorem}

We introduce some concepts from probability used in the proof of Theorem~\ref{thm:frobenius-preservation}. Let $X$ be a random variable. For $p \geq 1$, the $L^{p}$ norm is defined as $\|X\|_{L^{p}} = ( \mathbb{E} | X |^{p} )^{1/p}$. Let $\alpha > 0$, $X$ is an $\alpha$--sub-exponential random variable if for $t > 0$, $\text{Pr} \left\{ | X | \geq t \right\} \leq c \,\text{exp}(-C t ^{\alpha} ) $, where $c$ and $C$ are two absolute constants. The quasi-norm $\| \cdot \|_{\Psi_{\alpha}}$ is defined as
\[ \| X \|_{\Psi_{\alpha}} = \text{inf} \left\{ t > 0 \; | \; \mathbb{E} \text{exp}\left( | X |^{\alpha}/t^{\alpha} \right) \leq 2 \right\}. \]
\begin{proof}[Proof of Theorem~\ref{thm:frobenius-preservation}]
We can easily prove that  the estimator is unbiased by showing that $\mathbb{E}\left[\frac{1}{\ell} \|\mathbf{A} \mathbf{\Omega}\|_F^2\right] = \|\mathbf{A}\|_F^2$. The proof of the remaining part is split into two steps. In the first step, we define the sequence of independent random variables $X_i$ for  $1 \le i\le \ell$, and some properties of the random variables. In the second step, we establish concentration inequalities for the sum of the random variables.

\paragraph{Step 1: Defining the random variables}Define the sequence of independent random variables 
\begin{align*}
    X_i &:=  \frac{1}{\|\Mx{A}\|_F^2}\left( \| \Mx{A} \Vc{\omega}^{(i)}\|^2 - \|\Mx{A}\|_F^2\right)= \frac{1}{\|\Mx{A}\|_F^2} \sum_{j=1}^{m} \left( (\Vc{a}_j \Vc{\omega}^{(i)})^2 - \| \Vc{a}_j\|_2^2 \right),
\end{align*}
for $1\le i \le \ell$, where $\Vc{a}_{j} \in \R^{1 \times N}$ denotes the $j$-th row of $\Mx{A}$ and $\Vc{\omega}^{(i)} \in \R^{N \times 1}$ represent the $i$-th column of $\Mx{\Omega}$.
Next, we will analyze the properties of $X_i$ for $1 \le i \le \ell$. 

Let $\Vc{\omega} = \Vc{\omega}_{1} \kron \cdots \kron \Vc{\omega}_{d}$ where $\Vc{\omega}_{k} \in \R^{n_k\times 1 }$ for $k =\crly{1, \ldots, d}$ with \textit{i.i.d.} entries sampled from $\mathcal{N}(0,1)$. For any arbitrary vector $\Vc{y}_k \in \R^{1 \times n_k}$, by~\cite[Lemma 3.4.2]{vershynin2018high}, $\Vc{y}_k\Vc{\omega}_k$ is subgaussian with norm  $\|\Vc{y}_k\Vc{\omega}_k \|_{\Psi_2}\le C_S K \|\Vc{y}_k\|_2$, where $C_S$ is an absolute constant and $K$ is bound on subgaussian entries of $\Vc{\omega}^{(k)}$. For the standard normal distribution, $K=1$. Next, due to~\cite[Section 2.5.2]{vershynin2018high}, we have 
\[ \| \Vc{y}_k \Vc{\omega}_k\|_{L^p} \le C_S \sqrt{p} \|\Vc{y}_k\|_2 \qquad \forall ~\V{y}_k \in \R^{n_j}, \> 1 \le j \le d, \> p \ge 1. \] 
Using~\cite[Lemma 19]{ahle2020oblivious}, whose assumptions we just verified, we have that for any  arbitrary vector $\V{y} \in \R^{1 \times N}$ 
\begin{equation} \label{eqn:krd}\| \V{y}\Vc{\omega}\|_{L^p} \le (C_S \sqrt{p})^d \|\V{y}\|_2, \qquad \forall p \ge 1.   
\end{equation}
By using the triangle inequality and the centering property of $L^p$ norms~\cite[Lemma 47]{haselby2025fast},  we have
 \[ \begin{aligned}
 \|X_i\|_{L^p} &\leq  \frac{1}{\|\Mx{A}\|_F^2}\sum_{j=1}^m \left \| (\Vc{a}_j \Vc{\omega}^{(i)})^2 - \| \Vc{a}_j\|_2^2 \right\|_{L^p} \leq  \frac{2}{\|\Mx{A}\|_F^2} \sum_{j=1}^m \left \| (\Vc{a}_j \Vc{\omega}^{(i)})^2 \right\|_{L^p}\\
 & \leq  \frac{2}{\|\Mx{A}\|_F^2}\sum_{j=1}^m \left \| \Vc{a}_j \Vc{\omega}^{(i)} \right\|_{L^{2p}}^2\leq  \frac{2}{\|\Mx{A}\|_F^2} (C_s \sqrt{2p})^{2d} \|\Mx{A}\|_F^2 = 2 \, (C_s \sqrt{2})^{2d}  \, p^d.    
 \end{aligned}\]
As $\|X_i\|_{L^{p}} \leq 2 (C_s \sqrt{2})^{2d} p^d$ implies $X_i$ is an $\alpha$--sub-exponential random variable with $\alpha =\frac{1}{d}$. More precisely, by~\cite[Lemma A.2]{gotze2021concentration} 
\[ \|X_i\|_{\Psi_\alpha} \le (2e)^{d} 2(C_S \sqrt{2})^{2d}  \equiv C_{d}.  \]
\paragraph{Step 2: Concentration inequality} Let $\Vc{b} := \frac{1}{\ell}(C_d, \ldots, C_d)$ be an $\ell$ dimensional vector, then $\|\Vc{b}\|_{2} = C_d/\sqrt{\ell}$ and $\|\Vc{b}\|_{\infty} = C_d/\ell$. For $\alpha = 1/d < 1$, by \cite[Theorem 3.1]{kuchibhotla2022moving}, we have
$$
\text{Pr} \left\{ \left|\frac{1}{\ell}\sum_{i=1}^\ell X_i \right| \geq C_{\alpha}' \frac{\sqrt{t}}{ \sqrt{\ell}} +  C_{\alpha}'' \frac{t^d}{\ell}\right\} \leq 2e^{-t}   \quad \text{ for } t \geq 0
$$
 where the constants $C_{\alpha}' \equiv 2 e C(\alpha) C_d$,  $C_{\alpha}'' \equiv \sqrt{2}e 4^d  C(\alpha) C_d $  and $$C(\alpha) \equiv 4e^3(2\pi)^{1/4} e^{1/24} (e^{2/e}/\alpha)^{1/\alpha}$$ depend only on $\alpha = 1/d$. Setting  the failure probability $2 e^{-t} = \delta$ implies $t = \ln(2/\delta)$. Thus, with probability at least $1-\delta$, we have
 $$\left|\frac{1}{\ell}\sum_{i=1}^\ell X_i \right| \leq C_{\alpha}' \sqrt{\frac{\ln(2/\delta)}{\ell}} +  C_{\alpha}'' \frac{\ln^d(2/\delta)}{\ell}.
 $$
 Setting $\varepsilon  \geq C_{\alpha}' \sqrt{\frac{\ln(2/\delta)}{\ell}} +  C_{\alpha}'' \frac{\ln^d(2/\delta)}{\ell}$, we get an inequality $\ell \geq \beta + \gamma \sqrt{\ell}$, where $\beta = \varepsilon^{-1} C_{\alpha}'' \ln^d(2/\delta) $ and $\gamma = \varepsilon^{-1} C_{\alpha}' \ln(2/\delta)$. We can solve the equality $\ell = \beta + \gamma \sqrt{\ell}$ for positive root $\sqrt{\ell}$ as
 $$
 \sqrt{\ell}_{+} = \gamma/2 + \sqrt{\gamma^2 + 4 \beta}/2.
 $$
The function $\ell - \gamma \sqrt{\ell} - \beta$ is a parabola that is concave upward in $\ell$ and crossing the real line at the root $\sqrt{\ell}_{+}$.
 For $\sqrt{\ell} \geq \sqrt{\ell}_{+}$, $\ell - \gamma \sqrt{\ell} - \beta \geq 0$, so we can take $\sqrt{\ell} \geq 2 \max\{\sqrt{\beta}, \gamma\}$ or $\ell \geq 4 \max\{\beta, \gamma^2\}$. This gives the minimal number of samples $\ell$ as stated in the theorem. 
\end{proof}
The above theorem gives a relative error bound for the performance of the norm estimation. The corollary below is more useful to get an order-of-magnitude estimate of the norm. 

\begin{corollary}\label{corr:frobeniusnormpreservation}
    Let $\Mx{A}$ and $\Mx{\Omega}$ be defined as in Theorem~\ref{thm:frobenius-preservation}. Let $\zeta >  1 $ be a user-defined parameter and $ 0 < \delta < 1$ be the failure probability.  If the number of samples $\ell$ satisfies $$\ell \geq 4\max\left( C_\alpha'' \frac{\zeta}{\zeta-1} \ln^d(2/\delta ), \frac{\zeta^2}{(\zeta-1)^2} {C_{\alpha}'}^2 \ln^2(2/\delta)\right),$$ where $C_{\alpha}'$ and $C_{\alpha}''$ are constants depending only on $\alpha \equiv 1/d$,  then with probability at least $1-\delta$
\[  \frac{1}{\zeta} \|\Mx{A}\|_F^2 \le \frac{1}{\ell}\|\Mx{A} \Mx{\Omega}\|_F^2 \le \zeta\|\Mx{A}\|_F^2 .
\]
\end{corollary}
\begin{proof}
    Since $\zeta > 1$, set $1 - \varepsilon = \zeta^{-1} < 1$. From Theorem~\ref{thm:frobenius-preservation},  with probability at least $1-\delta$
    \[ \zeta^{-1} \|\Mx{A}\|_F^2 =    (1-\varepsilon) \|\Mx{A}\|_F^2 \le \frac{1}{\ell}\|\Mx{A} \Mx{\Omega}\|_F^2 \le (1+\varepsilon) \|\Mx{A}\|_F^2 \le \zeta \|\Mx{A}\|_F^2. \] 
The last inequality follows since $\varepsilon > 0$, so $1 - \varepsilon^2 < 1$ and $1 + \varepsilon < \frac{1}{1 - \varepsilon} = \zeta$. This completes the proof.
\end{proof}

\paragraph{Numerical illustration} Using the above theorem, we can approximate the norm of a tensor $\Tn{X} \in \R^{n_1 \times \cdots \times n_d}$ by setting $\Mx{A} = \Mx{X}_{(1)}$ and $\Mx{\Omega} = \Mx{\Omega}_{d} \kr \cdots \kr \Mx{\Omega}_{2}$. Figure~\ref{fig:norm_estimaton} summarizes the empirical evaluation of norm estimation of the order-$d$ synthetic tensors. The procedure for generating the synthetic tensors is described in Section~\ref{sec:synthetic_tensor}.

\begin{figure}[!ht]  
    \centering
    \includegraphics[width=0.5\linewidth]{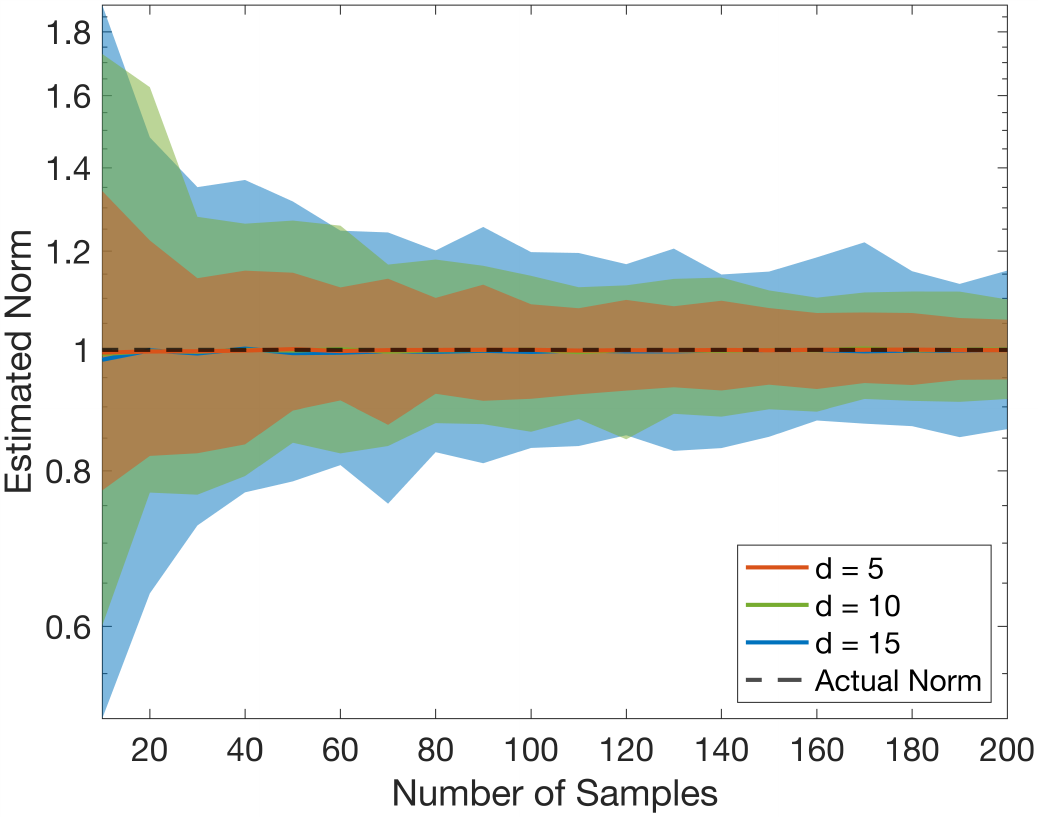}
   \caption{Average estimated norm over 1000 runs for synthetic tensors with different numbers of modes ($d$). The solid lines represent the average, and the shaded region indicates the minimum and maximum values observed across the runs.}

    \label{fig:norm_estimaton}
\end{figure}
In Figure~\ref{fig:norm_estimaton}, we see that the accuracy of the norm estimates improves with an increasing number of samples, as expected. Furthermore, the variability in the estimates, indicated by the width of the shaded region, increases with the number of modes. This reflects greater estimation uncertainty for higher-order tensors, which aligns with the theoretical results presented in Theorem~\ref{thm:frobenius-preservation}.

\paragraph{Residual error estimation} Theorem~\ref{thm:frobenius-preservation} and Corollary~\ref{corr:frobeniusnormpreservation} give a justification for the initial norm estimation (see also Remark~\ref{rem:normestimation}). We now argue that these results provide a justification of the Frobenius norm estimate of the residual matrix used in Algorithm~\ref{alg:krp_ttroudning_adapRank}. Suppose in Algorithm~\ref{alg:krp_ttroudning_adapRank}, at iteration $k$, the basis matrix corresponding to mode-$k$ computed so far is denoted by $\Mx{Q}_k$. Then, the residual matrix, say $\Mx{A}_k$  corresponding to $\Mx{Q}_k$ is 
$$
\Mx{A}_k = \left(\Mx{I} - \Mx{Q}_k \Mx{Q}_k^T \right) \v{\Tn{Y}_{k}}\h{\Tn{X}_{k+1:d}} \in \R^{n_k\ell_{k-1}\times \prod_{j=k+1}^d n_{j}}.
$$
Let $\Mx{\Omega}^{(k)} = \Mx{\Omega}_d \krp \cdots \krp \Mx{\Omega}_{k+1}$, where $\Mx{\Omega}_j \in \R^{n_j \times b_{\rm inc}^{(k)}}$ are drawn  from Gaussian distribution for $j \in \{k+1, \ldots, d\}$. Then the sketch matrix $\Mx{S}_k \in \R^{\ell_{k-1} n_k \times b_{\rm inc}^{(k)}}$ of residual matrix $\Mx{A}_k$
 with respect to the random matrix $\Mx{\Omega}^{(k)}$ is updated as in~\eqref{eq:resid_sketch_b}:
 \begin{align*}
 \Mx{S}_k &\leftarrow \left(\Mx{I} - \Mx{Q}_k \Mx{Q}_{k}^T \right) \v{\Tn{Y}_{k}} \underbrace{\h{\Tn{X}_{k+1:d}} (\Mx{\Omega}_d \krp \cdots \krp \Mx{\Omega}_{k+1})}_{ \equiv\Mx{W}_{k+1}(:,\col{\Mx{Q}_k}+1 : \col{\Mx{Q}_k}+b_{\rm inc}^{(k)})}\\
 &= \underbrace{\left(\Mx{I} - \Mx{Q}_{k} \Mx{Q}_{k}^T \right)  \Mx{S}_{k}}_{\equiv \text{Equation \eqref{eq:resid_sketch_b}}}.
 \end{align*}
By replacing  $\Mx{A}$ with the residual matrix $\Mx{A}_k$, $\Mx{\Omega}$ with the sketching matrix $ \Mx{\Omega}^{(k)}$ and $d$ with $d-k$, Theorem~\ref{thm:frobenius-preservation} and Corollary~\ref{corr:frobeniusnormpreservation} provide a justification of the Frobenius norm estimate of the residual matrix used in Algorithm~\ref{alg:krp_ttroudning_adapRank}.

\section{Experimental Results} \label{sec:experiments}
 Given a tensor $\Tn{X}$ and its approximation $\hat{\Tn{X}}$, the relative error is defined to be ${\|\Tn{X} - \hat{\Tn{X}}\|}/{\|\Tn{X}\|}$. As mentioned earlier, code required to reproduce our results is available at {https://github.com/bhisham123/TT-Rounding-using-KRP-structure.git}.

\textbf{Hardware description:} We perform all the timing experiments using MATLAB version 2023b on a machine with the following configuration:  CPU: Intel Xeon Gold 6226R @ 64x 2.9 GHz, RAM: 251 GB, SWAP: 8.0 GB, OS: Ubuntu 20.04.6 LTS. 

\subsection{Synthetic Data}  \label{sec:synthetic_tensor}
In our first experiment, we compare the speed and accuracy of fixed-rank TT-rounding algorithms on synthetic tensors with a known low-rank structure. 
Here we establish that for this problem, where no adaptivity is required since the ranks are known \textit{a priori}, the use of the proposed KRP structured random matrices compared to the TT-structured random matrices (as in the Randomize-then-Orthogonalize approach of \cite{al2023randomized}) has only minor effects on the time to solution and accuracy of the result. 

\paragraph{Description of problem setup} To generate a random tensor with a fixed target rank, we follow the procedure described in \cite[Section 4.1]{al2023randomized}. Specifically, we construct a random tensor $\Tn{X}$ by perturbing a lower-rank tensor $\Tn{Y}$ with a scaled random tensor $\epsilon \Tn{Z}$, as follows: $\Tn{X} = \Tn{Y} + \epsilon \Tn{Z}$. The tensors $\Tn{Y}, \Tn{Z} \in \mathbb{R}^{100 \times \cdots \times 100}$ are order-$d$ tensors with $d = 10$, and both are normalized random TT-tensors. The TT-ranks of $\Tn{Y}$ and $\Tn{Z}$ are set to $(1, 50, \ldots, 50, 1)$. We set the perturbation parameter to $\epsilon = 10^{-5}$ in our experiments. This construction results in a perturbed tensor $\Tn{X}$ with TT-ranks $(1, 100, \ldots, 100, 1)$. The perturbation parameter $\epsilon$ determines how well the tensor $\Tn{X}$ is approximated by the lower rank tensor $\Tn{Y}$, and a small value of $\epsilon$ ensures that $\Tn{X}$ remains close to $\Tn{Y}$.

\paragraph{Description of algorithmic settings} We compress the tensor $\Tn{X}$ using several algorithms to a low-rank tensor $\hat{\Tn{X}}$ with TT-ranks $(1, \ell, \ldots, \ell, 1)$. For convenience, we refer to Algorithm~\ref{alg:randomize_then_orthogonalize} as Rand-Orth, Algorithm~\ref{alg:orthogonalize_then_randomize} as Orth-Rand (Fix), Algorithm~\ref{alg:TT-Rounding} as TT-Rounding, and Algorithm~\ref{alg:krp_ttroudning_fixRank} as Rand-Orth-KRP (Fix).
In our experiments, we vary the target rank $\ell$ from $35$ to $80$ in increments of $5$. 

We summarize the approximation error and speedup results in Figure~\ref{fig:synthetic}. The approximation error is computed as the relative norm error as defined before, and the speedup is measured with respect to the deterministic TT-Rounding algorithm.

\begin{figure}[!ht]
    \centering
    \includegraphics[width=0.9\linewidth]{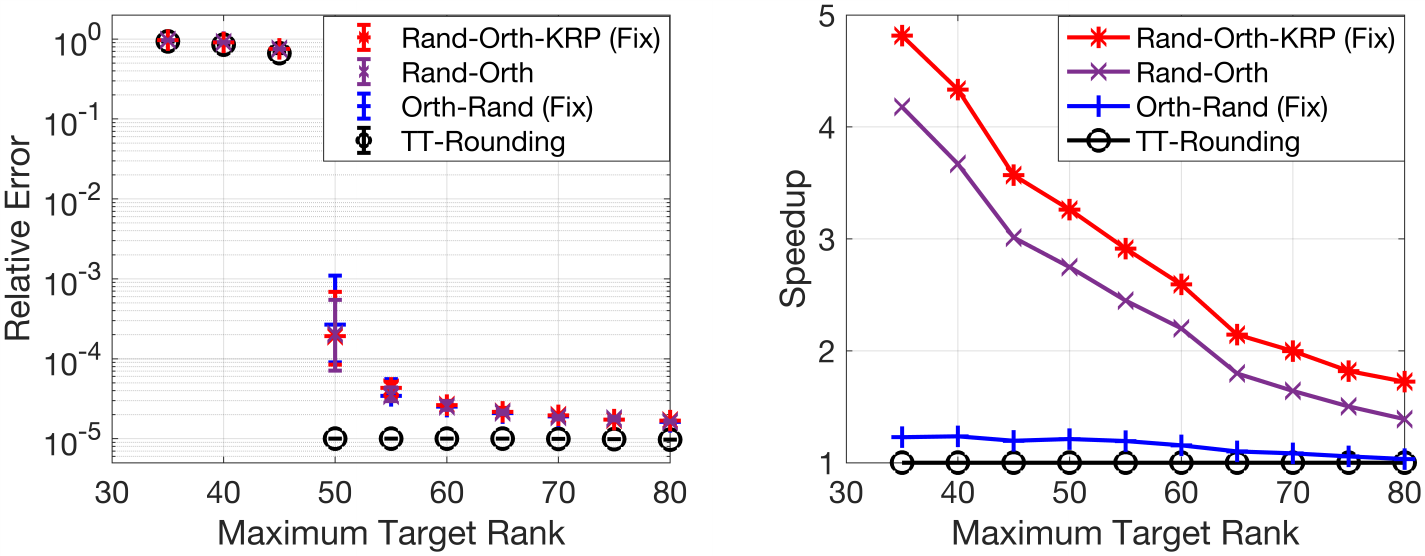}
    \caption{Left: Relative error over 10 runs on synthetic data. Right: Average speedup with respect to TT-Rounding algorithms over 10 runs.}
    \label{fig:synthetic}
\end{figure}

\paragraph{Reults and discussion} In Figure~\ref{fig:synthetic}, we can see that for $\ell < 50$, the relative error of all the algorithms is nearly identical and large. For $\ell \geq 50$, the relative error of deterministic TT-Rounding algorithms becomes close to  $10^{-5}$ as expected from the construction of $\Tn{X}$. The relative errors of all randomized algorithms, Rand-Orth, Orth-Rand (Fix), and Rand-Orth-KRP (Fix), are similar to each other but slightly higher than that of TT-Rounding. However, the randomized algorithms achieve significant speedups over TT-Rounding. The speedup of Rand-Orth-KRP (Fix) and Rand-Orth are comparable and are about $1.5$ to $4.8$ times faster than the TT-Rounding algorithm.

\subsection{Parametric low-rank kernel approximation} \label{ssec:matern}
Kernel matrices arise in many applications including machine learning, inverse problems, spatial statistics, etc. Kernel matrices are typically large and dense and computationally challenging to work with. To make it more challenging, the kernels depend on (hyper) parameters, so that the kernel matrices have to be computed repeatedly. 
In recent work~\cite{khan_parametric_2025} a method is proposed to construct a parametric low-rank kernel approximation that combines a polynomial approximation of the kernel with a TT compression of the coefficient matrix $\Tn{M}$. 
While~\cite{khan_parametric_2025} uses greedy-cross approximation to compress the tensor, the resulting tensor does not have optimal ranks. 
We propose to use the TT-rounding algorithms proposed in this paper. 
We show that for problems with a specified accuracy required (i.e., rank adaptivity), our proposed Algorithm~\ref{alg:krp_ttroudning_adapRank} computes approximations up to $10\times$ faster than deterministic and randomized alternatives while meeting the desired accuracy.

\paragraph{Description of problems settings} The specific problem setup is discussed in Section~\ref{ssec:app_matern}. 

\paragraph{Description of algorithmic settings} We compress the tensor $\Tn{X}$ at various accuracy levels using the following adaptive tolerance methods: Algorithm~\ref{alg:krp_ttroudning_adapRank} (denoted as {Rand-Orth-KRP (Adap)}), an adaptive version of the Orthogonalize-then-Randomize algorithm~\cite[Algorithm SM2.1]{al2023randomized} (denoted as {Orth-Rand (Adap)}), and the deterministic TT-rounding algorithm presented in Algorithm~\ref{alg:TT-Rounding}, referred to as TT-Rounding. In both Rand-Orth-KRP (Adap) and Orth-Rand (Adap), the enrichment of the orthonormal basis is guided by the Frobenius norm of the residual error matrix. Rand-Orth-KRP (Adap) uses an estimate of this norm rather than computing it exactly, due to the high computational cost of exact residual norm computation in this setting.
For Rand-Orth-KRP (Adap), we set $f_{\rm init} = 0.1$ and $f_{\rm inc} = 0.05$. In contrast, Orth-Rand (Adap) computes the Frobenius norm exactly, since the orthogonalization step enables the computation of the Frobenius norm with minimal additional work.

Since Rand-Orth-KRP (Adap) produces left-orthonormal compressed TT representations, excess ranks arising due to block size and residual norm estimation can be mitigated by applying a final compression pass to the orthonormal tensor cores, as suggested in Remark~\ref{rem:rounding_pass}. We refer to the resulting method as Adap-R.

We use the rank information from the compressed TT representations produced by the adaptive algorithms  to compress the Mat\'ern kernel tensor using their fixed-rank counterparts: {Orth-Rand-KRP (Fix)} (Algorithm~\ref{alg:krp_ttroudning_fixRank}) and Orth-Rand (Fix) (\cite[Algorithm SM2.1]{al2023randomized}).  Furthermore, we compress the Mat\'ern kernel tensor using Algorithm~\ref{alg:randomize_then_orthogonalize} (denoted as Rand-Orth) by utilizing the rank information of the compressed TT representation obtained using Rand-Orth-KRP (Adap) to compare its performance with Rand-Orth-KRP (Fix) under identical rank constraints. By giving an advantage to the fixed-rank approaches, this experiment will enable us to assess the trade-off in computational cost, particularly in runtime between the adaptive algorithms and their fixed-rank versions. 
\begin{figure}[!ht]
    \centering
    \includegraphics[width=0.98\linewidth]{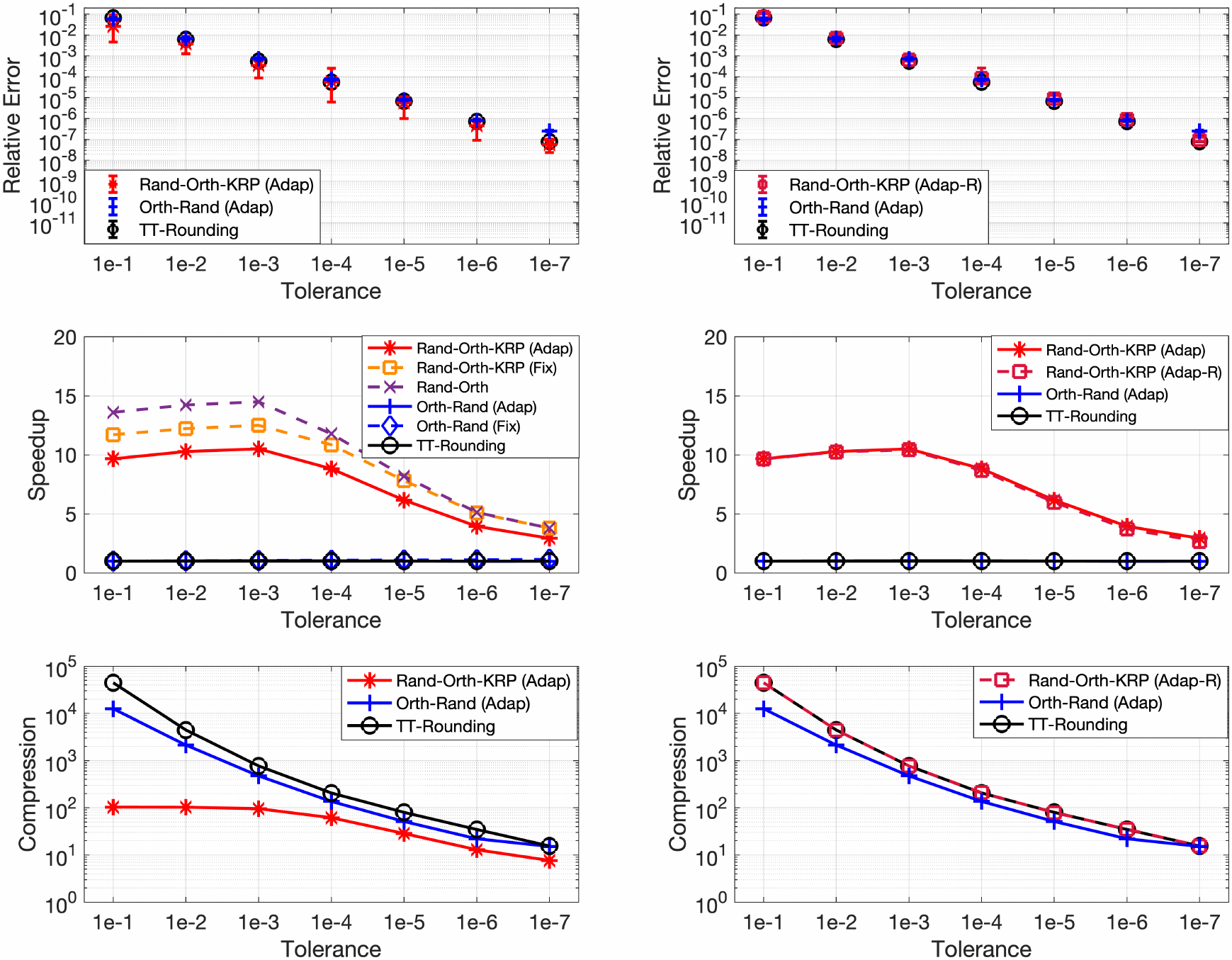}
   \caption{
Left column: Comparison of baseline methods on the Mat\'ern kernel data. 
Right column: Comparison after applying a rounding pass to the left-orthonormal TT tensor obtained from adaptive randomized TT rounding algorithms. 
Relative errors obtained over 10 independent runs. 
Average speedup is measured relative to the deterministic TT-Rounding algorithm. 
Compression is defined as the ratio of the number of parameters in the original tensor to that in the compressed tensor and is averaged over 10 runs. 
Here, {Adap} refers to adaptive-rank algorithms, {Fix} to their fixed-rank counterparts, and {Adap-R} to the adaptive-rank algorithm followed by a rounding pass on the left-orthonormal TT tensor.
}
    \label{fig:matern}
\end{figure}

\paragraph{Results and discussion} Figure~\ref{fig:matern} summarizes the experimental results obtained over 10 independent runs. In the upper-left panel of Figure~\ref{fig:matern}, the x-axis represents the desired accuracy level, while the y-axis shows the relative error between the compressed TT representation and the input TT tensor. We can see that the relative error achieved by Rand-Orth-KRP (Adap) is generally less than or equal to the desired accuracy. Only in a few cases, the error slightly exceeds the desired value but remains close. This behavior is likely due to the use of an estimate of the residual norm error. In contrast, the relative error of TT-Rounding closely matches the desired accuracy, as expected from a deterministic algorithm. Similarly, the relative error of  Orth-Rand (Adap) also closely matches the desired accuracy, since it computes the residual norm error exactly.

In the middle left panel of Figure~\ref{fig:matern}, we compare the runtime of the algorithms. Among the adaptive-rank methods, Rand-Orth-KRP (Adap) is the fastest, achieving up to $10\times$ speedup over both TT-Rounding and Orth-Rand (Adap). The Orth-Rand (Adap) algorithm shows no speed advantage over TT-Rounding because both involve a right-to-left orthogonalization pass, which dominates the overall computational cost. We can see that Rand-Orth-KRP (Fix) is slightly faster than Rand-Orth-KRP (Adap), as expected, since the adaptive variant includes additional orthogonalization (Line~\ref{line:ResSketch_orth} of Algorithm~\ref{alg:resdSketching}) and re-orthogonalization (Line~\ref{line:adap_reOrtho} of Algorithm~\ref{alg:krp_ttroudning_adapRank}) steps. However, the cost of these steps contributes only to lower-order terms in the overall complexity of Algorithm~\ref{alg:krp_ttroudning_adapRank}, and therefore does not significantly affect the performance of the adaptive-rank method compared to its fixed-rank counterpart, as evident from the middle-left subplot. Furthermore, we note that the speedups of Rand-Orth and Rand-Orth-KRP (Fix) are comparable across all desired accuracy levels.

The bottom left panel compares the average compression achieved by Rand-Orth-KRP (Adap), Orth-Rand (Adap), and TT-Rounding algorithms. Compression is defined as the ratio of the number of parameters in the original tensor to those in the compressed tensor. Among the three methods, TT-Rounding achieves the highest compression, closely followed by Orth-Rand (Adap). Rand-Orth-KRP (Adap) yields the lowest compression. This lower compression is due to the block-wise enrichment of the orthonormal basis and use of the residual norm estimate for stopping criteria,    which can increase TT-ranks beyond the minimal requirement.

The results after applying the rounding pass in Adap-R are summarized in the right column of the Figure~\ref{fig:matern}. Note that some of the results (TT-Rounding, Orth-Rand (Adap), and Rand-Orth-KRP (Adap) are repeated in the right panel to enable comparison). The three main takeaways are:  (1) the additional rounding pass ensures that the relative error is commensurate with the requested tolerance,   (2) the cost of the compression pass is negligible compared to the total runtime, as it operates on already compressed and orthogonalized cores, and (3) the additional compression results in lower ranks, and the compression ratios of Rand-Orth-KRP (Adap-R) is comparable with TT-Rounding.

\subsection{Parametric PDE in the TT format} \label{ssec:cookies}
We consider a parametric PDE usually referred to as the \textit{cookie problem}, \cite{kressner_low-rank_2011, tobler_christine_low-rank_2012}, in which the set of solutions is known to be low-rank. This PDE is discretized using finite differences and results in large linear system. We solve such a system using the TT-GMRES algorithm \cite{dolgov_tt-gmres_2012}. In TT-GMRES, the dominant computational bottlenecks occur during the repeated formation of linear combinations of TT-tensors—first in the operator application involving a sum over Kronecker products, and second during orthogonalization via Gram–Schmidt. Each of these steps requires a TT-rounding operation to control rank growth, making them critical to both the accuracy and efficiency of the overall method. 
We can exploit the structure of the input as a sum of TT tensors using randomization and call the operation Sum+Round (see Remark~\ref{rem:sum}). For further details on how this is done for the Rand-Orth algorithm, see \cite[Section 4.2]{al2023randomized}.

\paragraph{Description of problem setting}  The details can be found in Section~\ref{ssec:app_cookies}.
\begin{figure}[!ht]
    \centering
    \includegraphics[width=0.9\linewidth]{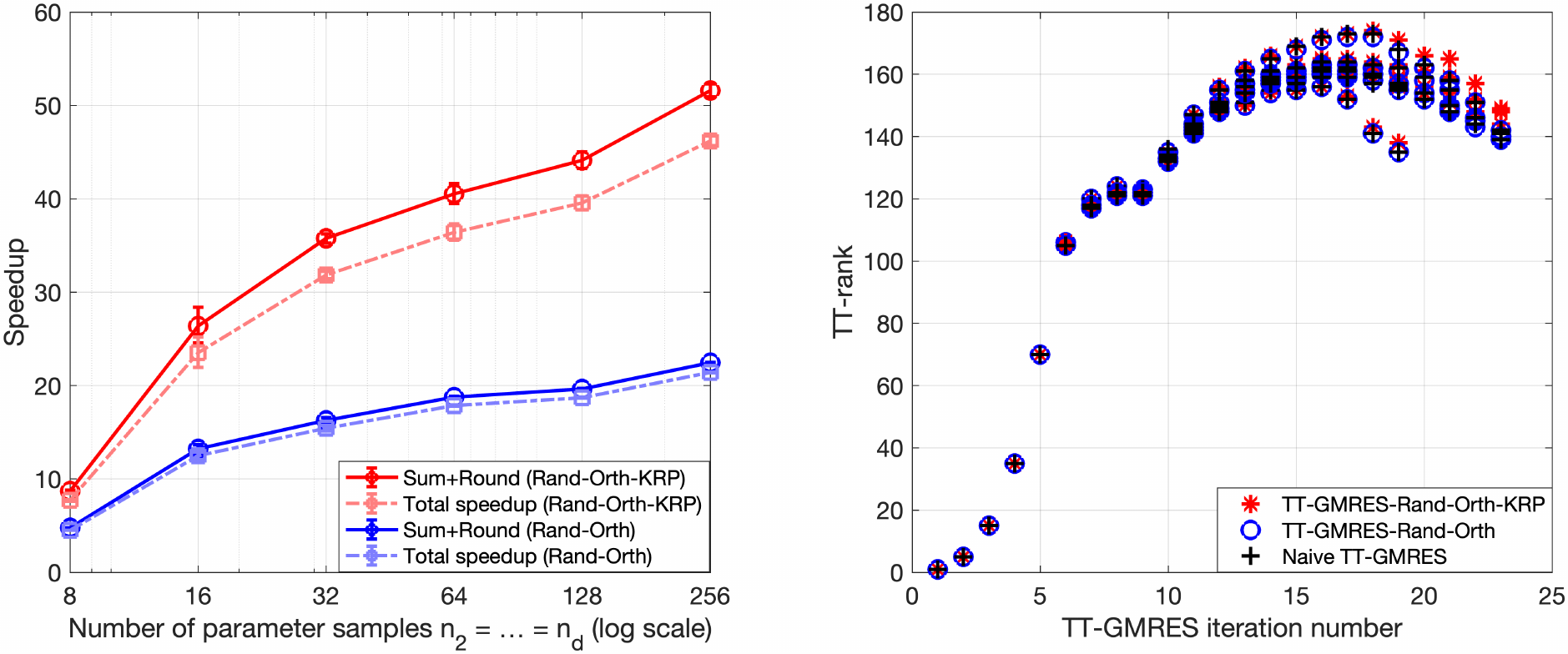}
    \caption{Left: Speedup relative to the naive TT-GMRES implementation on the parametric PDE  with 4 cookies. Right: Maximum TT-ranks at each iteration across 3 independent runs.}
    \label{fig:TTGMRES_speedup_rank}
\end{figure}
\paragraph{Description of algorithmic settings} For tensor summation followed by rounding within the preconditioned TT-GMRES solver, we compare three strategies: the naive deterministic implementation of preconditioned TT-GMRES, the randomized summation-then-orthogonalize approach from \cite{al2023randomized}, which we denote as {TT-GMRES-Rand-Orth}, and our proposed adaptive randomized method based on KRPs tailored to inputs that are sums of TT tensors (see Remark~\ref{rem:sum} and Algorithm~\ref{alg:krp_ttroudning_adapRank_ttsum}). We denote our proposal by {TT-GMRES-Rand-Orth-KRP}. In TT-GMRES-Rand-Orth-KRP, during the rounding of the sum of TT tensors (see Algorithm~\ref{alg:krp_ttroudning_adapRank_ttsum}), we set the initial block size for the $k$-th mode equal to the maximum $k$-th TT-rank among the summand tensors, and we set the enrichment factor $f_{\rm inc} = 0.05$. On the left-orthonormal compressed TT tensor obtained from Algorithm~\ref{alg:krp_ttroudning_adapRank_ttsum}, we perform a final deterministic rounding pass to reduce any excess rank that may arise due to the block size or the residual norm estimation (see Remark~\ref{rem:rounding_pass}). We set the relative tolerance in the TT-GMRES solver equal to  $10^{-8}$. All three solvers achieved similar relative residual errors, on the order of the specified solver tolerance.

\paragraph{Results and discussion} Figure~\ref{fig:TTGMRES_speedup_rank} summarizes the runtime speedup and the maximum TT-ranks per iteration for the different algorithms.
From the left panel in Figure~\ref{fig:TTGMRES_speedup_rank}, we observe that our method, TT-GMRES-Rand-Orth-KRP, is the fastest one. It achieves approximately a $51\times$ speedup in the summation-following-rounding operation relative to the naive TT-GMRES, and about a $46\times$ speedup in the overall TT-GMRES process. Furthermore, TT-GMRES-Rand-Orth-KRP is approximately $2\times$ faster than the previously best-performing method, TT-GMRES-Rand-Orth, proposed in~\cite{al2023randomized}. In the right panel, we see that the maximum TT-rank at each iteration remains nearly the same for all three algorithms across the three independent runs. This implies that our method maintains comparable rank growth behavior to the baseline algorithms while achieving significantly better speedup.

To assess the adaptivity of the TT-GMRES algorithms used in the above experiment, we perform an additional experiment on the cookie problem with a different configuration. We consider  $d = 8$ parameter dimensions and use a piecewise linear finite element discretization. The number of parameter samples is varied as $ n = n_2 = \cdots = n_d$, with values taken from the set $\{ 2^3, 2^4, \ldots, 2^8 \}$. The parameter values $\rho_i$ are chosen to be linearly spaced in the interval $[1, 5]$. We set the relative tolerance for the TT-GMRES solver to $10^{-8}$. Figure~\ref{fig:TTGMRES_error_new} presents a comparison of the relative residual error achieved by different algorithms. Figure~\ref{fig:TTGMRES_speedup_rank_new} summarizes the runtime and the maximum TT-ranks per iteration for the different algorithms. 
\begin{figure}[h]
    \centering
    \includegraphics[width=0.45\linewidth]{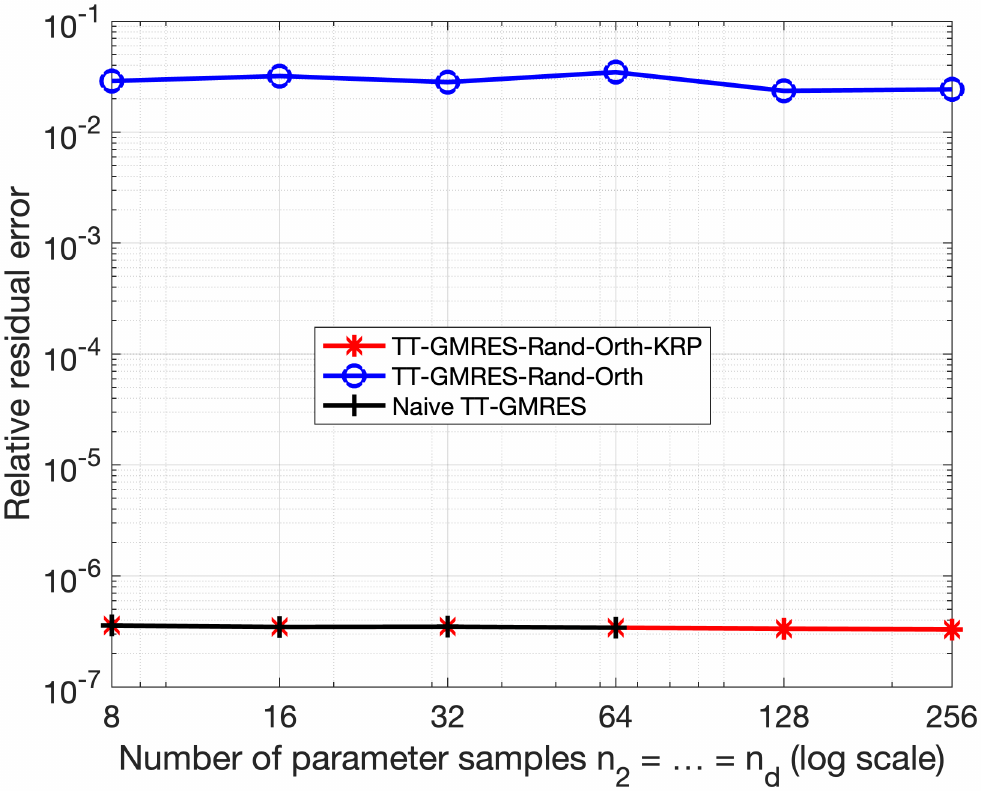}
    \caption{Average relative residual error achieved over 3 runs on the parametric PDE problem with 7 cookies. The na\"ive TT-GMRES approach encountered out-of-memory errors at sample sizes 128 and 256.}
    \label{fig:TTGMRES_error_new}
\end{figure}

\begin{figure}[h]
    \centering
    \includegraphics[width=0.92\linewidth]{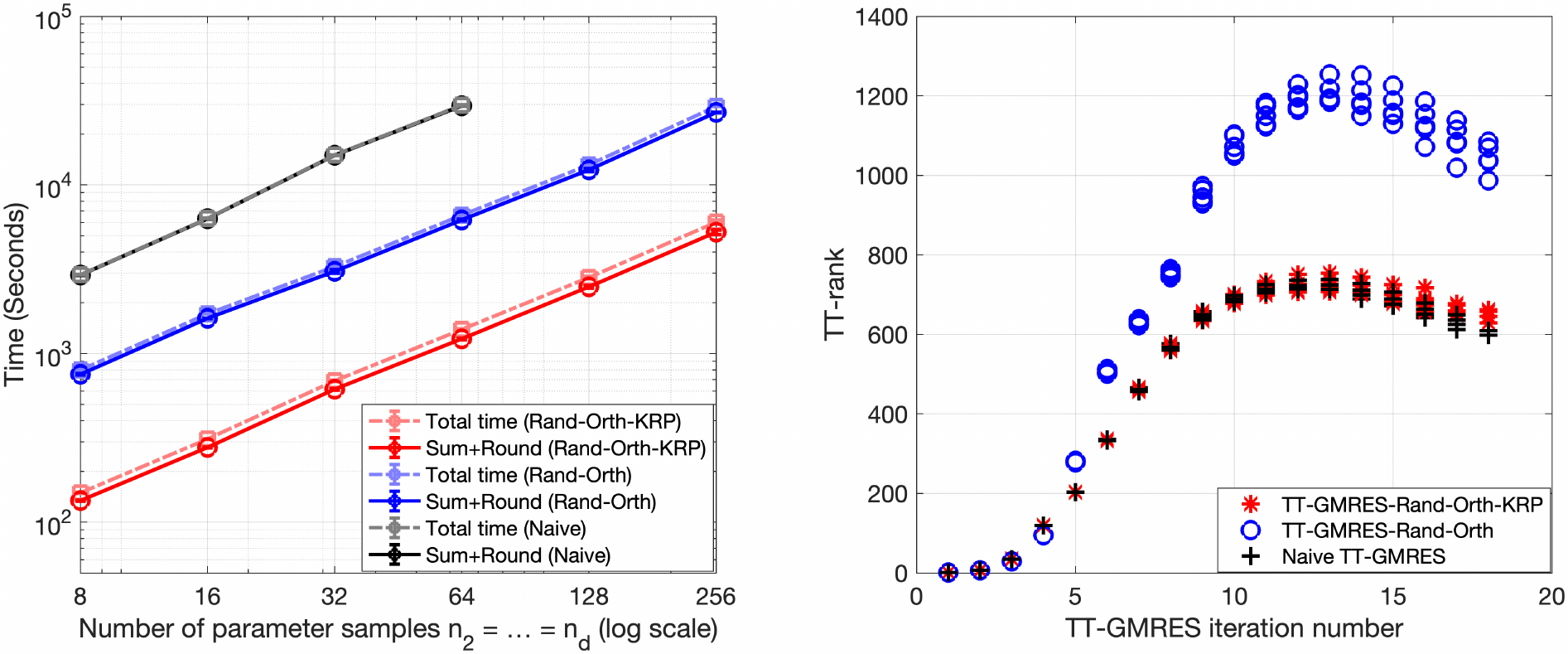}
    \caption{Left: Runtime comparison on the parametric PDE problem with 7 cookies. Right: Maximum TT-ranks at each iteration across 3 independent runs.}
    \label{fig:TTGMRES_speedup_rank_new}
\end{figure}

From Figure~\ref{fig:TTGMRES_error_new}, we observe that the TT-GMRES-Rand-Orth algorithm achieves a relative residual error of only $10^{-2}$, highlighting its inability to adapt effectively to the new dataset. In contrast, the randomized TT-GMRES-Orth-Rand-KRP and the naive TT-GMRES both achieve significantly lower residual errors, on the order of $10^{-7}$, which is much closer to the desired target of  $10^{-8}$. The slight deviation from the target accuracy may be due to the inexact Arnoldi procedure used in the GMRES process. In the left panel of Figure~\ref{fig:TTGMRES_speedup_rank_new}, we see that the na\"ive TT-GMRES method has the highest runtime and fails with out-of-memory errors at sample sizes 128 and 256. On the other hand, the randomized TT-GMRES-Orth-Rand-KRP method has the lowest runtime, approximately $25\times$ faster than the naive method and $5\times$ faster than the randomized TT-GMRES-Rand-Orth algorithm.
In the right panel,  we can see that the maximum TT-rank at each iteration remains nearly the same for the naive method and TT-GMRES-Rand-Orth-KRP algorithms, while the TT-GMRES-Orth-Rand algorithm consistently has significantly higher TT-ranks across the three independent runs. These results demonstrate that the TT-GMRES-Rand-Orth-KRP method preserves the rank behavior of the naive approach while delivering notable computational savings, making it a more practical and scalable alternative for large-scale problems.

\section{Conclusion} \label{sec:conclusion}
We introduce a new adaptive randomized TT-rounding algorithm where sketching employs matrices with KRP structure.
Underpinned with statistical analysis of the error estimation to provide adaptiveness, the proposed algorithm demonstrates comparable accuracy compared to the deterministic TT-SVD algorithm while achieving up to $50\times$ speedup on a number of test cases arising from synthetic data, parameter-dependent PDE, and parameter-dependent kernel matrices. 
Compared to previously proposed randomized fixed-rank algorithms, a slight overhead is required in order to guarantee a satisfactory accuracy as demonstrated in our numerical experiments.

We observe a huge benefit from using the randomized TT-rounding in order to compress a sum of TT tensors without assembling the sum prior to compression. 
This avoids a massive amount of data movement and therefore results in a fast fused-summation-and-compression algorithm allowing solvers such as TT-GMRES to be practical for tensor equations.

Another observation regarding sketching using matrices with KRP structure is that a significant part of computing the sketches can be performed in parallel with respect to the tensor cores. This can potentially have a significant impact on the derivation of parallel algorithms for TT rounding especially with the intrinsic sequential paradigm of TT-SVD to keep the complexity linear in the number of modes. We aim to address this in future work. Another area of future work is to extend our approach to exploit the structure of MPO-MPS compression as in~\cite{camano_successive_2025}.

\section*{Acknowledgments}
We would like to acknowledge the help of Abraham Khan in generating the test problem in Section~\ref{ssec:matern}.

\appendix
\section{Overview of Appendices}
In these supplementary materials, we provide more details about three algorithms and two application problems that appear in the numerical experiments of the main paper.
We discuss the pre-existing Orth-Rand and Rand-Orth algorithms for rounding a single TT input in Section~\ref{sec:otr} and Section~\ref{sec:rto} respectively, and how our proposed algorithm using KRP sketches can be tailored to sum of TT inputs in Section~\ref{sec:sum}.
The details of the problem setups are given in Section~\ref{sec:probsetup}.

\section{Orthogonalize then Randomize (Orth-Rand)}\label{sec:otr}

The Orthogonalize-then-Randomize algorithm is introduced in \cite{al2023randomized}. As the name suggests, it consists of two phases: an orthogonalization phase followed by the randomization phase. Let $\Tn{X}$ be the input TT tensor which we want to compress into tensor $\Tn{Y}$. The orthogonalization phase of the Orth-Rand algorithm is similar to the orthogonalization phase of deterministic TT rounding  (Algorithm~\ref{alg:TT-Rounding}), where we make the horizontal unfoldings of TT cores orthonormal starting from the rightmost core to the left, to obtain a right orthonormal tensor $\Tn{Y}$ which is equivalent to $\Tn{X}$.  See Section~\ref{sec:tt_rounding} for a detailed discussion on the orthogonalization phase. During the randomization phase, we process the TT cores of $\Tn{Y}$ from left to right, excluding the last core. For each vertical unfolding $\v{\Tn{Y}_k}$, we apply the randomized range finder to get orthonormal basis $\Mx{Q}_k \in \R^{\ell_{k-1} n_k \times \ell_k}$ such that
$$
\v{\Tn{Y}_k} \approx  \Mx{Q}_k \Mx{Q}_k^T \v{\Tn{Y}_k}
$$
We then set $\v{\Tn{Y}_k} = \Mx{Q}_k$ and $\h{\Tn{Y}_{k+1}} = \Mx{Q}_k \Mx{Q}_k^T \h{\Tn{Y}_{k+1}}$. For more details on the randomization phase, refer to \cite{al2023randomized}. 
The detailed pseudocode is given in Algorithm~\ref{alg:orthogonalize_then_randomize}.
This algorithm can be made adaptive to a prescribed error tolerance by applying a standard adaptive randomized low-rank matrix algorithm as described in \cite[Section SM2]{al2023randomized}.

\begin{algorithm}[!ht]
    \caption{TT-Rounding:Orthogonalize-then-Randomize \cite{al2023randomized}}
    \label{alg:orthogonalize_then_randomize}
  
\begin{algorithmic}[1]
    \Require A tensor $\Tn{X} = [ \Tn{X}_1, \ldots, \Tn{X}_d ]$ with TT-ranks $\{r_k\}$, target ranks$\{\ell_k\}$
    \Ensure A tensor $\Tn{Y} = [\Tn{Y}_1, \ldots, \Tn{Y}_d ]$  with TT-ranks $\{\ell_k\}$
    \Function{$\Tn{Y}$=TT-Rounding-OrthRand}{$\Tn{X}$, $\{\ell_k\}$}
    \State $\Tn{Y}$ = \textproc{OrthogonalizeRL}($\Tn{X}$)
    \For{$k=1$ to $d-1$ } 
    \State $\Mx{Z}_k$ = $\v{\Tn{Y}_k}$ \Comment{$\Tn{Y}_k$ is $\ell_{k-1} \times n_k \times r_k$}
    \State $\Mx{S}_k= \Mx{Z}_k \Mx{\Omega}_k$ \Comment{form sketched matrix $n_k \ell_{k-1} \times \ell_k$}
    \State [$\v{\Tn{Y}_k}$, $\sim$] = \texttt{QR}($\Mx{S}_k$) \Comment{reduced {QR} decomposition}
    \State $\Mx{M}_k  = \v{\Tn{Y}_k}^T \Mx{Z}_k$ \Comment{form $\ell_k \times r_k$ matrix} 
    \State $\h{\Tn{Y}_{k+1}} = \Mx{M}_k \h{\Tn{Y}_{k+1}}$ \Comment{$\Tn{Y}_{k+1} = \Tn{Y}_{k+1} \times_1 \Mx{M}_k$}
    \EndFor 
    \EndFunction
\end{algorithmic}

\end{algorithm}

\section{Randomize-then-Orthogonalize (Rand-Orth)}
\label{sec:rto}

Here, we present an overview of the Rand-Orth algorithm by \cite{al2023randomized}. The main idea of the Rand-Orth algorithm is to avoid the computationally expensive orthogonalization of the original TT cores by leveraging randomization techniques.  It replaces the costly orthogonalization step in line~\ref{line:tt_rounding_orth_step} of Algorithm~\ref{alg:TT-Rounding} with randomized sketching to reduce the computational overhead.
As the name suggests, the method consists of two phases: a randomization phase and an orthogonalization phase. The randomization phase involves the computation of partial contractions (also known as partial projections or sketches), Algorithm~\ref{alg:partialContractionRL_TT}. In this phase, we first generate a random Gaussian TT-tensor~\cite{al2023randomized} $\Tn{R}\in \mathbb{R}^{n_1 \times \cdots \times n_d}$, whose cores $\Tn{R}_{k}$, are filled with independent, normally distributed entries with zero mean and variance $1/(\ell_{k-1} n_k \ell_k)$, for $1 \leq k \leq d$.

Using this randomized tensor $\Tn{R}$,  we compute the partial contractions of the input tensor $\Tn{X}$ by taking a single pass over its TT-cores, either from left-to-right or from right-to-left. Algorithm~\ref{alg:partialContractionRL_TT} presents a right-to-left contraction; a left-to-right variant can be obtained by reversing the order of operations in the same procedure.
\begin{algorithm}[!ht]
  \caption{Right-to-Left Contraction \cite{al2023randomized}}
  \label{alg:partialContractionRL_TT}
  \begin{algorithmic}[1]
    \Require TT-tensors $\Tn{X} = [\Tn{X}_1, \ldots, \Tn{X}_d]$ and $\Tn{Y} = [ \Tn{Y}_1, \ldots, \Tn{Y}_d]$  with TT-ranks $\{r_k\}$ and $\{\ell_k\}$
    \Ensure Matrices $\{\Mx{W}_{k}\}$ satisfy $\Mx{W}_{k} = \h{\Tn{X}_{k+1:d}} \h{\Tn{Y}_{k+1:d}}^T$
    \Function{[\{$\Mx{W}_{k}$\}]=PartialContractionsRL}{$\Tn{X}$, $\Tn{Y}$}
    \State $\Mx{W}_{d-1} = \h{\Tn{X}_d} \h{\Tn{Y}_d}^T$
    \For{$k=d-1$ to $2$ }
    \State $\v{\Tn{Z}_k} = \v{\Tn{X}_k} \Mx{W}_k$ \Comment{$\Tn{Z}_{k} = \Tn{X}_k \times_3 \Mx{W}_k$}
    \State $\Mx{W}_{k-1} = \h{\Tn{Z}_k} \h{\Tn{Y}_k}^T$ \Comment{$\Mx{W}_{k-1}$ is $r_{k-1} \times \ell_{k-1}$}
    \EndFor
    \EndFunction
\end{algorithmic}

\end{algorithm}
In the orthogonalization phase, the TT-cores are traversed in the direction opposite to that used during the randomization phase. If the randomization phase proceeds from right-to-left, the orthogonalization is performed from left-to-right, and vice versa. This reverse traversal facilitates the efficient computation of the final sketches on which the orthogonalization is performed, ultimately producing a compressed left-orthogonal (right-orthogonal) TT-tensor. Starting with $k = 1$ and initializing $\Tn{Y}_1 = \Tn{X}_1$, we proceed with the computation of the QR factorization of the sketched matrix 
\begin{align}
     \v{\Tn{Y}_k} \h{\Tn{X}_{k+1:d}} \h{\Tn{R}_{k+1:d}}^T = \v{\Tn{Y}_k} \Mx{W}_k = \Mx{Q}_k \Mx{R}_k.
\end{align}
At the $k$-th step, the first $k - 1$ cores of $\Tn{Y}$ have already been orthogonalized and, therefore, do not need to be considered explicitly in the subsequent factorization. We approximate the product of the final $d-k+1$ core by projecting the product $\v{\Tn{Y}_k} \h{\Tn{X}_{k+1:d}}$ on the column space of $\Mx{Q}_k$ as 
\begin{align}
\v{\Tn{Y}_k} \h{\Tn{X}_{k+1:d}} \approx \Mx{Q}_k \Mx{Q}_k^T  \v{\Tn{Y}_k} \h{\Tn{X}_{k+1:d}} = \Mx{Q}_k \Mx{M}_k \h{\Tn{X}_{k+1:d}}.
\end{align}
 where $\Mx{M}_k$ is the resulting factor from the projection. The TT-cores are then updated accordingly: $\v{\Tn{Y}_k} \leftarrow \Mx{Q}_k$, and $\h{\Tn{Y}_{k+1}} \leftarrow \Mx{M}_k \h{\Tn{X}_{k+1}}$. A complete pseudo-code of this procedure is provided in Algorithm~\ref{alg:randomize_then_orthogonalize}, and for more details refer to \cite{al2023randomized}.
\begin{algorithm}
    \caption{TT-Rounding: Randomize-then-Orthogonalize \cite{al2023randomized}}
    \label{alg:randomize_then_orthogonalize}
  \begin{algorithmic}[1]
    \Require A tensor $\Tn{X} = [\Tn{X}_1, \ldots, \Tn{X}_d]$ with TT-ranks $\{r_k\}$, target ranks$\{\ell_k\}$
    \Ensure A tensor $\Tn{Y} =  [\Tn{Y}_1, \ldots, \Tn{Y}_d]$  with TT-ranks $\{\ell_k\}$
    \Function{$\Tn{Y}$=TT-Rounding-RandOrth}{$\Tn{X}$, $\{\ell_k\}$}
    \State Select a random Gaussian TT-tensor $\Tn{R}$ with target ranks $\{\ell_k\}$
    \State $\{\Mx{W}_k\}$ = \textproc{PartialContractionsRL}($\Tn{X}$, $\Tn{R}$)
    \State $\Tn{Y}_1 = \Tn{X}_1$
    \For{$k=1$ to $d-1$ } \label{line:comp_start}
    \State $\Mx{Z}_k$ = $\v{\Tn{Y}_k}$ \Comment{$\Tn{Y}_k$ is $\ell_{k-1} \times n_k \times r_k$}
    \State $\Mx{S}_k= \Mx{Z}_k \Mx{W}_k$ \Comment{form sketched matrix $n_k \ell_{k-1} \times \ell_k$}
    \State [$\v{\Tn{Y}_k}$, $\sim$] = \texttt{QR}($\Mx{S}_k$) \Comment{reduced QR decomposition}
    \State $\Mx{M}_k  = \v{\Tn{Y}_k}^T \Mx{Z}_k$ \Comment{form $\ell_k \times r_k$ matrix} 
    \State $\h{\Tn{Y}_{k+1}} = \Mx{M}_k \h{\Tn{X}_{k+1}}$ \Comment{$\Tn{Y}_{k+1} = \Tn{X}_{k+1} \times_1 \Mx{M}_k$}
    \EndFor \label{line:comp_end}
    \EndFunction
\end{algorithmic}

\end{algorithm}

\section{Rounding of the sum of TT Tensor}\label{sec:sum}

The summation of TT tensors is a common operation in TT arithmetic, and it inherently increases the ranks of the resulting tensor. Hence, it often relies on TT rounding, where the goal is to compress a tensor $\Tn{X}$ that is given as a sum of $s$ TT tensors: $\Tn{X} = \Tn{X}^{(1)} + \cdots + \Tn{X}^{(s)}$. A naive approach to compressing the sum of TT tensors is to first form their formal sum, resulting in a TT tensor $\Tn{X}$ with ranks equal to the sums of the individual TT ranks and cores exhibiting a sparse block structure. Then round $\Tn{X}$ using rounding algorithms. However, this method is computationally expensive, exhibiting cubic scaling in the number of summands $s$ when using classical TT rounding (Algorithm~\ref{alg:TT-Rounding}), and quadratic scaling for randomized methods (Algorithms~\ref{alg:randomize_then_orthogonalize},~\ref{alg:krp_ttroudning_fixRank},~\ref{alg:krp_ttroudning_adapRank}). As $s$ grows large, this becomes computationally infeasible in practice.

To handle this, we can exploit the property of the randomized sketching algorithm that the sketch of the sum is the sum of the sketches as suggested in \cite{al2023randomized}. That is, we can sketch each TT summand individually using Algorithm~\ref{alg:partialContractionRL_KRP} and then accumulate the result efficiently. We omit the details of this case because it follows the approach described in \cite[Section 3.3]{al2023randomized}; the only difference for the fixed-rank case is in how the partial contractions are computed to exploit the KRP structure.

This idea naturally extends to the adaptive-rank variant (Algorithm~\ref{alg:krp_ttroudning_adapRank}), which aims to maintain a prescribed relative error when the target TT ranks are not known a priori.  The detailed pseudocode of this extension is presented in the Algorithm~\ref{alg:krp_ttroudning_adapRank_ttsum}.

\begin{algorithm}[!ht]
    \caption{Residual sketching for basis expansion for the sum of TT tensor using Khatri–Rao projections.}
     \label{alg:resdSketching_ttsum}
    \begin{algorithmic}[1]
    \Require TT tensors $\{\Tn{X}^{(i)}\}_{i=1}^s$, matrix $\Mx{Z} \in \R^{\ell_{k-1} n_k \times \sum_{i=1}^s r_{k}^{(i)}}$, basis matrix $\Mx{Q}$, partial contractions $\{\Mx{W}_{j}^{(i)}\}_{\substack{j=1, \ldots,d\\i = 1, \ldots, s}}$,  TT core index $k$ and block size $b$
    \Ensure Residual sketch $\Mx{S}$ and partial contraction matrices $\{\Mx{W}_{j}^{(i)}\}_{\substack{j=1, \ldots,d\\i = 1, \ldots, s}}$ appended with new columns 
    \Function{[$\Mx{S}$, $\{\Mx{W}_j^{(i)}\}$]=GenerateResidualSketch-TTSum}{$\{\Tn{X}^{(i)}\}$, $\Mx{Z}$, $\Mx{Q}$, $\{\Mx{W}_{j}^{(i)}\}$, $k$, $b$}
    \If {$\col{\Mx{W}_{k+1}^{(1)}} < \col{\Mx{Q}} + b$} \Comment{$\col{\cdot}$ returns  number of columns} 
        \State Select random matrices $\Mx{\Omega}_{k+1}, \ldots, \Mx{\Omega}_d$ of size $n_j \times (b - \col{\Mx{W}_{k+1}^{(1)}})$ for \hspace*{9.5cm} $j =\crly{k+1, \ldots, d}$ 
        \For{$i = 1$ to $s$}
            \State $\{\overline{\Mx{W}}_j^{(i)}\}$ = \textproc{KRP-PartialContractionsRL}($\Tn{X}^{(i)}$, $[\Mx{\Omega}_{k+1}, \ldots, \Mx{\Omega}_d]$)
             \For {$j = k+1$ to $d$}
                \State $\Mx{W}_{j}^{(i)} = \bmat{\Mx{W}_j^{(i)} &  \overline{\Mx{W}}_{j}^{(i)}}$  \Comment{append columns}
            \EndFor
        \EndFor
    \EndIf
        \State  $\Mx{S}$ = $\Mx{Z}
       \begin{bmatrix}
        \Mx{W}_{k+1}^{(1)}\left(:,\col{\Mx{Q}}+1: \col{\Mx{Q}} +b \right)\\
        \vdots\\
        \Mx{W}_{k+1}^{(s)}\left(:,\col{\Mx{Q}}+1: \col{\Mx{Q}} +b \right)
       \end{bmatrix}$
        \State $\Mx{S} = \Mx{S} - \Mx{Q} \left(\Mx{Q}^T \Mx{S} \right)$  \label{line:ResSketch_ttsum_orth} \Comment{orthogonalization of $\Mx{S}$}
    \EndFunction
\end{algorithmic}

\end{algorithm}

\begin{algorithm}[!ht]
    \caption{Adaptive Rank TT-Rounding for sum of TT tensors: Randomize-then-Orthogonalize using KRP}
     \label{alg:krp_ttroudning_adapRank_ttsum}
  \begin{algorithmic}[1]
   \small
    \Require TT-tensors $\{\Tn{X}^{(i)}\}_{i=1}^s$ with TT-ranks $\{r_k^{(i)}\}$, tolerance $\varepsilon$,  $0 < f_{inc} < 1$ to set increment block size
    \Ensure A tensor $\Tn{Y}= [ \Tn{Y}_1, \ldots, \Tn{Y}_d ]$  with $\|\Tn{X} - \Tn{Y}\|/\|\Tn{X}\| \leq \varepsilon$ with high probability.
   
    \Function{$\Tn{Y}$=TT-Rounding-RandOrth-KRP-Adaptive}{$\{\Tn{X}^{(i)}\}$, $\varepsilon$, $f_{inc}$}
    \State Select random Gaussian matrices $\Mx{\Omega}_2, \ldots, \Mx{\Omega}_d$ of size $n_k \times r$ for $1 \leq k \leq d-1$ where $r = \max(r_{j}^{(i)})$ for $1\leq j \leq d-1$ and $1\leq i \leq s$ 
    \For{i = 1 to s}
    \State \{$\Mx{W}_k^{(i)}$\} = \textproc{KRP-PartialContractionsRL}($\Tn{X}^{(i)}$, $[\Mx{\Omega}_2, \ldots, \Mx{\Omega}_d]$)  
    \EndFor    
    \State $\Tn{Y}_1$ = $\begin{bmatrix}
        \Tn{X}_1^{(1)} & \cdots & \Tn{X}_1^{(s)} 
    \end{bmatrix}$
    \State $\Mx{S} = \v{\Tn{Y}_{1}} \begin{bmatrix}
        \Mx{W}_1^{(1)} \\ \vdots \\ \Mx{W}_1^{(s)} 
    \end{bmatrix}$
     \State $\tau = \varepsilon \cdot \left(\|\Mx{S}\|_F/\sqrt{r} \right)/\sqrt{d-1}$
    \For{$k=1$ to $d-1$ }
        \State $b_{init}^{(k)} =\max(\{r_{k}^{(i)}\}_{i=1}^s)$ \Comment{initial block size}
       \State $\Mx{S}_k = \v{\Tn{Y}_k }\begin{bmatrix}
        \Mx{W}_k^{(1)}(:,1:b_{init}^{(k)}) \\ \vdots \\ \Mx{W}_k^{(s)} (:,1:b_{init}^{(k)})
    \end{bmatrix}$
        \State $[\Mx{Q}_k,\sim]$ = $\texttt{QR}(\Mx{S}_k)$  \Comment{reduced QR decomposition}
        \State $\begin{bmatrix}
            \Mx{M}_k^{(1)}& \ldots & \Mx{M}_k^{(s)}
        \end{bmatrix} = \Mx{Q}_k^T \v{\Tn{Y}_{k}}$ 
        \State $\begin{bmatrix}
            \Mx{Y}^{(1)}, \ldots, \Mx{Y}^{(s)}
        \end{bmatrix} =
        \begin{bmatrix}
            \Mx{M}_k^{(1)} \h{\Tn{X}_{k+1}^{(1)}}, \ldots, \Mx{M}_k^{(s)} \h{\Tn{X}_{k+1}^{(s)}}  
        \end{bmatrix}$  \label{line:adp_step1_end_ttsum}
        \State $\h{\Tn{Y}_{k+1}} = \begin{bmatrix}\Mx{Y}^{(1)}& \ldots & \Mx{Y}^{(s)}\end{bmatrix}$
        \State $b_{inc}^{(k)} = \lceil \bar{r}_k \cdot f_{inc} \rceil$  \Comment{incremental block size, here $\bar{r}_k = \min(\size{\v{\Tn{Y}_k}})$}
        \State [$\Mx{S}_k$, $\{\Mx{W}_j^{(i)}\}$] = \textproc{GenerateResidualSketch-TTSum}($\{\Tn{X}^{(i)}\}$, $\v{\Tn{Y}_k}$, $\Mx{Q}_k$, $\{\Mx{W}_j^{(i)}\}$, $k$, $b_{inc}^{(k)}$)
        \While {$\|\Mx{S}_{k}\|_F/\sqrt{b_{inc}^{(k)}} \, > \, \tau$}  
                \State $[\Mx{Q}_k^{new},\sim] = \texttt{QR}(\Mx{S}_k)$  \Comment{reduced QR decomposition}
                \State $\Mx{Q}_k^{new} = \texttt{QR}\left(\Mx{Q}_k^{new} - \Mx{Q}_k \left(\Mx{Q}_k^T \Mx{Q}_k^{new} \right) \right)$ \label{line:adap_reOrtho_ttsum} \Comment{re-orthogonalization }
                \State $\Mx{Q}_k= \bmat{\Mx{Q}_k & \Mx{Q}_k^{new}}$ \Comment{add new orthonormal directions}
                \State $\begin{bmatrix}
                    \Mx{M}_k^{(1)}& \ldots & \Mx{M}_k^{(s)}
                \end{bmatrix} = \left(\mathbf{Q}_k^{new}\right)^T \v{\Tn{Y}_{k}} $ 
                \State $\begin{bmatrix}
                    \Mx{Y}^{(1)}, \ldots, \Mx{Y}^{(s)}
                \end{bmatrix} = 
                \begin{bmatrix}
                   \big[\Mx{Y}^{(1)}; \Mx{M}_k^{(1)} \h{\Tn{X}_{k+1}^{(1)}}\big],   \ldots,  \big[\Mx{Y}^{(s)}; \Mx{M}_k^{(s)} \h{\Tn{X}_{k+1}^{(s)}} \big]  
                \end{bmatrix}$  
                \State $\h{\Tn{Y}_{k+1}} = \begin{bmatrix} \Mx{Y}^{(1)} & \ldots & \Mx{Y}^{(s)} \end{bmatrix}$
        \State [$\Mx{S}_k$, $\{\Mx{W}_j^{(i)}\}$] = \textproc{GenerateResidualSketch-TTSum}($\{\Tn{X}^{(i)}\}$, $\v{\Tn{Y}_k}$, $\Mx{Q}_k$, $\{\Mx{W}_j^{(i)}\}$, $k$, $b_{inc}^{(k)}$)
        \EndWhile
        \State $\v{\Tn{Y}_{k}} = \Mx{Q}_{k}$
    \EndFor
    \EndFunction
\end{algorithmic}
\end{algorithm}

\section{Problem setup}\label{sec:probsetup}
In this section, we give additional details of the problem setups in the Numerical Experiments (Section~\ref{sec:experiments}). 
\subsection{Problem setup: Parametric kernel low-rank approximation}\label{ssec:app_matern}
In this section, we give some of the details of the problem setup in Section~\ref{ssec:matern}. 

Given two sets of points in \(\mathbb{R}^3\), \( X = \{ x_1, \ldots, x_{I_s} \} \) and \( Y = \{ y_1, \ldots, y_{I_t} \} \), called source and target points respectively, and a kernel function \(\kappa (\cdot, \cdot; \theta) : \mathbb{R}^3 \times \mathbb{R}^3  \to \mathbb{R}\), a kernel matrix \( K(X, Y;\theta) \in \mathbb{R}^{I_s \times I_t} \) has entries
\[
[K(X,Y;\theta)]_{i,j} = \kappa(x_i, y_j; \theta) \qquad 1 \le i \le I_s, 1 \le j \le I_t.
\]
The (hyper)parameters $\theta \in \mathbb{R}^p$ control the behavior of the kernel. 

In this application, we focus on the Mat\'ern kernel. Given parameters \(\theta = (\ell, \nu)\), where \(\ell\) is the length scale and \(\nu\) controls the smoothness of the kernel, the Mat\'ern kernel function is defined as
\[
\kappa(x, y; \theta) = \frac{2^{1-\nu}}{\Gamma(\nu)} \left( \sqrt{2 \nu} \frac{r}{\ell} \right)^\nu \mathcal{K}_\nu \left( \sqrt{2 \nu} \frac{r}{\ell} \right),
\]
where \( \mathcal{K}_\nu \) is the modified Bessel function of the second kind and $r= \|x-y\|_2$. This kernel is an example of an isotropic kernel.

We consider the kernel function \(\kappa(x, y; \theta)\) defined over a source box \(  C_s = [0, 1]^3 \) and a target box \( C_t = [2, 3]^3 \), where the source and target point sets \( X \subseteq C_s \) and \( Y \subseteq C_t \). The kernel parameters \(\theta = (\ell, \nu)\) vary over a 2-dimensional parameter space, with the length scale \(\ell\) sampled over the interval \([c \cdot D_b, D_b]\) and the smoothness \(\nu\) over \([0.5, 3]\). Here, 
\(
D_b = \| (b,b,b) - (0,0,0) \|_2 = \sqrt{3} \, b \approx 0.173,
\)
where \( b = 0.1 \) defines the extent of the parameter domain, and \( c = 0.5 \) controls the lower bound of \(\ell\). Thus, the kernel (hyper)parameters satisfy $\theta \in \Theta \subset  C_\theta $, where $C_\theta= [c \cdot D_b, D_b] \times [0.5, 3]$.

In this application, we only focus on the compression of the coefficient tensor $\Tn{M}$, as defined in~\cite[Section 2.4, below equation (2.8)]{khan_parametric_2025}. Essentially, it involves the kernel evaluations at Chebyshev grid points of the first kind on domain $C_s \times C_\theta \times C_t$. Note that the compression of $\Tn{M}$ is an important step in PTTK that affects many downstream computations. The respective domains $C_s, C_t, C_\theta$ are discretized using \( n = 100 \) Chebyshev nodes of the first kind in each spatial dimension. This results in an 8-way coefficient tensor $\Tn{M}$ (3 each for spatial variables $x$ and $y$, and $2$ for the parameter variables) with each mode size $n=100$. 
This high-dimensional tensor is then compressed using a greedy TT-cross approximation with a relative Chebyshev norm tolerance of \(10^{-12}\), resulting in a compact tensor-train (TT) representation of the parameterized Mat\'ern kernel. We call the resultant tensor-train representation the {Mat\'ern kernel tensor} $\Tn{X}$ and use it as input to the baseline algorithms in our experimental study.

\subsection{Problem Setup: Parametric PDE in the TT format}\label{ssec:app_cookies}

Consider a square domain $\Omega$ with $d$ mutually disjoint discs $D_{i}$ aligned on a grid. Solving the following PDE for all parameter values $\boldsymbol{\rho} = \{ \rho_{i} \}_{i = 1}^{d}$ simultaneously is a $d+2$ dimensional problem

\begin{align*}
    -\text{div}( \sigma(x, y; \boldsymbol{\rho} ) \nabla( u(x, y; \boldsymbol{\rho}) ) &= f(x,y) & \text{ in  } \Omega \\
    u(x, y; \boldsymbol{\rho}) &= 0 & \text{ on  } \delta \Omega,
\end{align*}
where, for $1 \le i \le d$,
\begin{align*}
    \sigma(x, y; \boldsymbol{\rho}) = \begin{cases}
        1 + \rho_{i} & (x,y) \in D_{i} \\
        1 & \text{ elsewhere}.
    \end{cases}
\end{align*}

Let $\Mx{A}_{1,1}$ be the operator discretized over the domain with constant parameter values, for $2 \leq i \leq d$, $\Mx{A}_{i,1}$ the discretization of the operator over the domain times the characteristic function corresponding to the subdomain, $\Mx{A}_{i,i}$ the diagonal matrix with parameter values, $\Mx{A}_{i,j}$ the identity matrix for $j \neq i$. Then we can consider the linear system encapsulating the parametric PDE

\[ \left( \sum_{i = 1}^{d} \Mx{A}_{i,1} \otimes \ldots \otimes \Mx{A}_{i, d} \right) \Tn{U} = \Tn{F},\]
where $\Tn{U}$ is the discretized solution tensor representing $u(x, y; \boldsymbol{\rho})$ over the spatial and parameter domains, and $\Tn{F}$ is the discretized right-hand side tensor representing $f(x,y)$.

In our numerical experiments, we follow the setup from \cite[Section 4.2]{al2023randomized} with $d = 5$ parameter dimensions and a piecewise linear finite element discretization. We vary the number of parameter samples $n = n_2 = \cdots = n_d$ over the values $\{ 2^3, 2^4, \ldots, 2^8 \}$, with $\rho_i$ values linearly spaced between 1 and 10.

\bibliographystyle{siam}
\bibliography{reference}
\end{document}